\documentclass[reqno,a4paper]{amsart}

\usepackage[latin1]{inputenc}
\usepackage[english]{babel}
\usepackage{eucal,amsfonts,amssymb,amsmath,amsthm,epsfig,mathrsfs}
\usepackage{cancel,soul}
\usepackage{color}
\usepackage{amscd,amsxtra}
\usepackage{enumerate}
\usepackage{latexsym}
\usepackage{bm}

\allowdisplaybreaks

\newcounter{ipotesi}

\Alph{ipotesi}
 \makeatletter \@addtoreset{equation}{section}

\makeatother \makeatletter

\newtheorem{thm}{Theorem}[section]
\newtheorem{hyp}[thm]{Hypotheses}{\rm}
{\rm}
\newtheorem{lemm}[thm]{Lemma}
\newtheorem{cor}[thm]{Corollary}

\newtheorem{prop}[thm]{Proposition}
\newtheorem{defi}[thm]{Definition}
\newtheorem{rmk}[thm]{Remark}{\rm}

\newcounter{parentenv}

\newcommand{\R}{{\mathbb R}}
\newcommand{\CC}{{\mathbb C}}
\newcommand{\E}{{\mathbb E}}
\newcommand{\N}{{\mathbb N}}

\newcommand{\X}{{\mathcal{X}}}
\newcommand{\D}{{\mathcal{D}}}
\newcommand{\J}{{\mathcal{D}}}

\newcommand{\eps}{\varepsilon}
\newcommand{\ra}{\rightarrow}

\newcommand{\Tr}{{\operatorname{Tr}}}
\newcommand{\tr}{{\operatorname{Tr}}}
\newcommand{\Dom}{{\operatorname{Dom}}}

\newcommand{\Span}{{\operatorname{span}}}
\newcommand{\Id}{{\operatorname{Id}}}

\newcommand{\set}[1]{{\left\{#1\right\}}}
\newcommand{\pa}[1]{{\left(#1\right)}}
\newcommand{\sq}[1]{{\left[#1\right]}}
\newcommand{\gen}[1]{{\left\langle #1\right\rangle}}
\newcommand{\abs}[1]{{\left|#1\right|}}
\newcommand{\norm}[1]{{\left\|#1\right\|}}
\newcommand{\scal}[2]{{\left\langle #1,#2\right\rangle}}

\newcommand{\eqsys}[1]{{\left\{\begin{array}{ll}#1\end{array}\right.}}
\newcommand{\tc}{\, \middle |\,}

\begin{document}

\frenchspacing

\title[On generators of transition semigroups associated to semilinear SPDEs]{On generators of transition semigroups associated to semilinear stochastic partial differential equations}

\author[D. A. Bignamini]{{D. A. Bignamini}}

\address[D. A. Bignamini]{Dipartimento di Scienze Matematiche, Fisiche e Informatiche, Universit\`a degli Studi di Parma, Parco Area delle Scienze 53/A, 43124 Parma, Italy.}
\email{\textcolor[rgb]{0.00,0.00,0.84}{davideaugusto.bignamini@unimore.it}}

\author[S. Ferrari]{{S. Ferrari}$^*$}\thanks{$^*$Corresponding Author}

\address[S. Ferrari]{Dipartimento di Matematica e Fisica ``Ennio De Giorgi'', Universit\`a del Salento, Via per Arnesano snc, 73100 Lecce, Italy.}
\email{\textcolor[rgb]{0.00,0.00,0.84}{simone.ferrari@unisalento.it}}



\date{\today}

\begin{abstract}
Let $\X$ be a real separable Hilbert space. Let $Q$ be a linear, bounded, positive and compact operator on $\X$ and let $A:\Dom(A)\subseteq\X\ra\X$ be a linear, self-adjoint operator generating a strongly continuous semigroup on $\X$. Let $F:\X\ra\X$ be a (smooth enough) function and let $\{W(t)\}_{t\geq 0}$ be a $\X$-valued cylindrical Wiener process. For any $\alpha\geq 0$, we are interested in the mild solution $X(t,x)$ of the semilinear stochastic partial differential equation
\begin{gather*}
\eqsys{
dX(t,x)=\big(AX(t,x)+F(X(t,x))\big)dt+ Q^{\alpha}dW(t), & t>0;\\
X(0,x)=x\in \X,
}
\end{gather*}
and in its associated transition semigroup 
\begin{align}\label{Tropical}
P(t)\varphi(x):=\E[\varphi(X(t,x))], \qquad \varphi\in B_b(\X),\ t\geq 0,\ x\in \X;
\end{align}
where $B_b(\X)$ denotes the space of the real-valued, bounded and Borel measurable functions on $\X$. In this paper we study the behavior of the semigroup $P(t)$ in the space $L^2(\X,\nu)$, where $\nu$ is the unique invariant probability measure of \eqref{Tropical}, when $F$ is dissipative and has polynomial growth. Then we prove the logarithmic Sobolev and the Poincar\'e inequalities and we study the maximal Sobolev regularity for the stationary equation
\[\lambda u-N_2 u=f,\qquad \lambda>0,\ f\in L^2(\X,\nu);\]
where $N_2$ is the infinitesimal generator of $P(t)$ in $L^2(\X,\nu)$.
\end{abstract}

\maketitle

\section{Introduction}

Let $(\Omega,\mathcal{F},\{\mathcal{F}_t\}_{t\geq 0},\mathbb{P})$ be a normal filtered probability space. We denote by $\E[\cdot]$ the expectation with respect to $\mathbb{P}$. Let $\X$ be a real separable Hilbert space with inner product $\scal{\cdot}{\cdot}$ and associated norm $\norm{\cdot}$. Let $A:\Dom(A)\subseteq\X\ra\X$ be the infinitesimal generator of a strongly continuous semigroup $e^{tA}$, $Q\in\mathcal{L}(\X)$ a positive operator and  $F: \X \ra \X$ be a suitable Borel measurable function. Let $\{W(t)\}_{t\geq 0}$ be a $\X$-valued adapted cylindrical Wiener process on $(\Omega,\mathcal{F},\{\mathcal{F}_t\}_{t\geq 0},\mathbb{P})$. For $\alpha\geq 0$, we consider the stochastic partial differential equation
\begin{gather}\label{eqFO}
\eqsys{
dX(t,x)=\big(AX(t,x)+F(X(t,x))\big)dt+Q^\alpha dW(t), & t>0;\\
X(0,x)=x\in \X.
}
\end{gather}
Under suitable hypotheses on $A$, $Q$ and $F$, for any $x\in\X$, \eqref{eqFO} has a unique mild solution $\set{X(t,x)}_{t\geq 0}$ (see Proposition \ref{contmild}) and the transition semigroup on $B_b(\X)$ (the space of the real-valued, bounded and Borel measurable functions on $\X$) defined as
\begin{align}\label{transition}
P(t)\varphi(x):=\E[\varphi(X(t,x))],\qquad t\geq 0,\ x\in \X,\ \varphi\in B_b(\X);
\end{align}
admits a unique probability invariant measure $\nu$ (see Proposition \ref{Davied}). Mild solutions of stochastic partial differential equations such as \eqref{eqFO} and their transition semigroups are widely studied in the literature (see, for example, \cite{AD-BA-MA1,BF20,CER1,CE-DA1,DA-ZA2,DA-ZA1,ES-ST1,ES-ST2,FA-GO-Z1,GO-GO1}).
 
In this paper we study the measure $\nu$, the behaviour of the transition semigroup \eqref{transition} in the space $L^2(\X,\nu)$ and the regularity of the weak solution of the stationary Kolmogorov equation associated to \eqref{eqFO}. In particular we will show that, whenever $A$, $Q$ and $F$ satisfy some additional conditions (see Hypotheses \ref{hyp1}), then $\nu$ satisfies the logarithmic Sobolev and Poincar\'e inequalities, while the transition semigroup $P(t)$ satisfies a hypercontractivity property. The behavior of the transition semigroup in the space $L^2(\X,\nu)$ and of the invariant measure $\nu$ is already present in the literature under some different hypotheses (see for example \cite{DA4,DA2,DA1,DA3,DA-LU2,DA-TU1,DA-TU2,DA-ZA1}). In order to be able to state our results and show how they differ from those already present in literature we need to list the hypotheses we will assume throughout the paper.
\begin{hyp}\label{hyp0} 
We assume the following conditions hold true.
\begin{enumerate}[\rm(a)]
\item $Q\in\mathcal{L}(\X)$ is positive and compact operator.

\item $A:\Dom(A)\subseteq \X\ra\X$ is linear, self-adjoint and is the infinitesimal generator of a strongly continuous semigroup $e^{tA}$ on $\X$. Moreover there exists $\zeta_1>0$ such that for every $x\in \Dom(A)$
\[
\scal{Ax}{x}\leq -\zeta_1\norm{x}^2.
\]

\item There exists $\eta\in (0,1)$, such that
\begin{align}\label{contconvu}
\int^{+\infty}_0\frac{1}{s^\eta}\Tr[e^{2sA}Q^{2\alpha}]ds<+\infty,
\end{align}

\end{enumerate}
\end{hyp}

\noindent Our hypotheses on $F$ are the following.

\begin{hyp}\label{hyp0.5}
Assume Hypotheses \ref{hyp0} hold true and that $F:\X\ra\X$ is a locally Lipschitz function, such that

\begin{enumerate}[\rm(a)]
\item \label{hyp0.5A} $F$ has polynomial growth, namely, there exist $m\in\N$ and $C>0$, such that for any $x\in\X$
\[
\norm{F(x)}\leq C(1+\norm{x}^m);
\]

\item \label{hyp0.5B} there exists $\zeta_2\in\R$ such that $\zeta:=\zeta_1-\zeta_2>0$ and, for any $x,y\in\X$,
\[
\scal{F(x)-F(y)}{x-y}\leq \zeta_2\norm{x-y}^2.
\]
\end{enumerate}
\end{hyp} 
Hypotheses \ref{hyp0.5} will be used the prove existence and uniqueness of the mild solution of \eqref{eqFO} and the fact that formula \eqref{transition} defines a semigroup.

\begin{rmk}
If Hypotheses \ref{hyp0.5} hold true, then the class of stochastic differential equations defined by \eqref{eqFO} covers two important cases. Indeed if $Q$ is a trace class operator, $\alpha=1/2$ and $A=-(1/2)\Id_{\X}$ we obtain the stochastic partial differential equation associated with a perturbed Malliavin operator (see for example \cite{AN-FE-PA2, BOGIE1,CAP-FER1,CAP-FER2,FA-GO-Z1}). While if $\alpha=0$ and $A$ is the pseudoinverse of a trace class operator we obtain the stochastic partial differential equation considered in \cite{DA1,DA-ZA2,DA-ZA1,DA-ZA4,FA-GO-Z1}.
\end{rmk}

Now we want to give an outline of the results of the paper and of the procedure we want to follow to prove such results. Our first step is to show that the semigroup $P(t)$, defined in \eqref{transition}, can be uniquely extended in $L^2(\X,\nu)$ to a strongly continuous semigroup $P_2(t)$ (see Proposition \ref{L22}). We will denote by $N_2$ the infinitesimal generator of $P_2(t)$ in $L^2(\X,\nu)$. 
The first main result we will prove is that the space
\[
\xi_A(\X):=\Span\{\mbox{real and imaginary parts of the functions } x\mapsto e^{i\scal{x}{h}}\,|\, h\in \Dom(A) \},
\]
is a core in $L^2(\X,\nu)$ for $N_2$. Moreover we will show that the closure in $L^2(\X,\nu)$ of the second order Kolmogorov operator defined by
\begin{equation}\label{OPFO}
 N_0\varphi(x):=\frac{1}{2}\tr[Q^{2\alpha}\D^2\varphi(x)]+\scal{x}{A\D\varphi(x)}+\scal{F(x)}{\D\varphi(x)},\qquad \varphi\in \xi_A(\X),\ x\in\X.
\end{equation}
is $N_2$. This first step is essential to the study of the weak solution of the stationary equation
\begin{align}\label{Saiaka}
\lambda u- N_2 u=f,
\end{align}
with $\lambda>0$ and $f\in L^2(\X,\nu)$, since it provides us with a ``easy to work with'' space on which we know the exact action of the operator $N_2$. More precisely the results we will prove in Section \ref{CLOS} is the following.
\begin{thm}\label{identif}
Assume Hypotheses \ref{hyp0.5} hold true. $N_2$ is the closure of $N_0$ in $L^2(\X,\nu)$ and $\xi_A(\X)$ is a core for $N_2$ in $L^2(\X,\nu)$.
\end{thm}
\noindent A proof of Theorem \ref{identif}, when $F$ is a Lipschitz continuous function, can be found in \cite[Section 3.5]{DA1}, \cite[Section 11.2.2]{DA-ZA1} or \cite{GO-KO1}, while in \cite[Sections 4.6 and 5.7]{DA1} Theorem \ref{identif} was proved in other specific settings. 

In the second step we want to provide an appropriate definition of Sobolev space related to a ``natural'' derivative operator associated to $Q^\alpha$. We start with some definitions that will be useful throughout the paper. Let $H_{\alpha}:=Q^{\alpha}(\X)$ and set
\begin{align*}
\scal{h}{k}_\alpha:=\langle Q^{-\alpha}h,Q^{-\alpha}k\rangle,\qquad h,k\in H_{\alpha}.
\end{align*}
\noindent $(H_{\alpha},\scal{\cdot}{\cdot}_\alpha)$ is a Hilbert space continuously embedded in $\X$. We denote by $\norm{\cdot}_{\alpha}$ the norm induced by the scalar product $\scal{\cdot}{\cdot}_\alpha$ on $H_\alpha$. 
\begin{defi}
Let $Y$ be a Hilbert space endowed with the norm $\norm{\cdot}_Y$ and let $\Phi: \X\rightarrow Y$. 
\begin{enumerate}[\rm(i)]
\item We say that $\Phi$ is differentiable along $H_\alpha$ at the point $x\in\X$, if there exists $L_x\in\mathcal{L}(H_\alpha,Y)$ such that 
\begin{align*}
\lim_{\norm{h}_\alpha\ra 0}\frac{\norm{\Phi(x+h)-\Phi(x)-L_xh}_Y}{\norm{h}_\alpha}=0.
\end{align*}
When it exists, the operator $L_x$ is unique and we set $\J_\alpha\Phi(x):=L$. If $Y=\R$, then $L_x\in H_\alpha^*$ and so there exists $k_x\in H_\alpha$ such that $L_xh=\gen{h,k_x}_\alpha$ for any $h\in H_\alpha$. We set $\D_\alpha\Phi(x):=k_x$ and we call it $H_\alpha$-gradient of $\Phi$ at $x\in\X$. 

\item We say that $\Phi$ is two times differentiable along $H_\alpha$ at the point $x\in \X$ if it is differentiable along $H_\alpha$ at every point of $\X$ and there exists $T_x\in\mathcal{L}(H_\alpha,\mathcal{L}(H_\alpha,Y))$ such that
\begin{align*}
\lim_{\norm{k}_\alpha\ra 0}\frac{\norm{(\J\Phi(x+k))h-(\J\Phi(x))h-(T_xh)k}_{Y}}{\norm{k}_\alpha}=0,
\end{align*}
uniformly for $h\in H_\alpha$ with norm $1$. When it exists, the operator $T_x$ is unique and we set $\J^2_\alpha\Phi(x):=T_x$. If $Y=\R$, then $T_x\in\mathcal{L}(H_\alpha,H_\alpha^*)$, so there exists $S_x\in\mathcal{L}(H_\alpha)$ such that $(T_xh)(k)=\gen{S_xh,k}_\alpha$, for any $h,k\in H_\alpha$. We set $\D^2_\alpha\Phi(x):=S_x$ and we call it $H_\alpha$-Hessian of $\Phi$ at $x\in\X$.
\end{enumerate}
\end{defi}
In a similar way it is possible to define the $H_\alpha$-Gateaux derivative.
\begin{rmk}
The derivative operator $\D_\alpha$ is a Fr\'echet derivative along the direction of $H_\alpha$. As we will show in Section \ref{supersob} the existence of the classical Fr\'echet derivative operator implies the existence of $\D_\alpha$, but the converse is false. Derivative operators along a dense subspace were already considered in papers and books dealing with Sobolev spaces defined on infinite dimensional spaces (see for example \cite{AN-FE-PA1,ASvN13,BF20,BOGIE1,CAP-FER1,CAP-FER2,GGvN03}).
\end{rmk}

We need to add some hypotheses in order to prove the rest of the results of the paper.

\begin{hyp}\label{hyp1} 
We assume Hypotheses \ref{hyp0.5} hold true, that $A$ and $Q$ commute and that there exists $G:\X\ra\X$ such that $F=Q^{2\alpha}G$. Furthermore we assume that 
\begin{enumerate}[\rm (a)]
\item $F$ is a Fr\'echet differentiable function with continuous derivative, such that there exist $C'>0$ with
\[
\|\J F(x)\|_{\mathcal{L}(\X)}\leq C'(1+\norm{x}^{m-1}),\qquad x\in\X;
\]
where $\J F$ denotes the Fr\'echet derivative operator of $F$.

\item  The part of $A$ in $H_\alpha$ generates a strongly continuous and contraction semigroup on $H_\alpha$ (we still denote by $A$ the part of $A$ in $H_\alpha$) and there exist $\zeta_\alpha>0$ such that
\[
\scal{(A+\J F(x))h}{h}_\alpha\leq -\zeta_\alpha\norm{h}_\alpha^2,\quad x\in \X\;h\in H_\alpha.
\] 
\end{enumerate}
\end{hyp}

\noindent In Section \ref{bloom} we are going to prove that the operators $\D_\alpha:\xi_A(\X)\subseteq L^2(\X,\nu)\ra L^2(\X,\nu;H_\alpha)$ and $(\D_\alpha, \D_\alpha^2):\xi_A(\X)\subseteq L^2(\X,\nu)\ra L^2(\X,\nu;H_\alpha)\times L^2(\X,\nu;\mathcal{L}(H_\alpha))$ are closable, and we introduce the Sobolev spaces $W_\alpha^{1,2}(\X,\nu)$ and $W_\alpha^{2,2}(\X,\nu)$ as the domains of their respective closures.

\begin{rmk}
In the case when $Q$ is a trace class opearator and $\alpha=0$ or $\alpha=1/2$ these Sobolev spaces introduced here were already considered and studied in various papers. Indeed $W^{1,2}_0(\X,\nu)$ is the Sobolev space studied in \cite[Section 3.6.1]{DA1} and \cite{DA-DE-GO1}, while the space $W^{1,2}_{1/2}(\X,\nu)$ is the Sobolev space studied in \cite{CAP-FER1} and \cite{FER1}. Furthermore we stress that 
\[
W^{1,2}_\alpha(\X,\nu)\subseteq W^{1,2}_\beta(\X,\nu),\qquad \alpha<\beta.
\]
In particular the norms $\norm{\cdot}_{W^{1,2}_\alpha(\X,\nu)}$, with $\alpha\geq 0$, are not equivalent. 
\end{rmk}

The main result of Section \ref{bloom} is the following.

\begin{thm}\label{Stime}
Assume Hypotheses \ref{hyp1} hold true. If $\lambda>0$ and $f\in L^2(\X,\nu)$, then \eqref{Saiaka} admits a unique weak solution $u\in \Dom(N_2)$. Furthermore $u\in W^{1,2}_\alpha(\X,\nu)$ and
\begin{gather*}
\norm{u}_{L^2(\X,\nu)}\leq\frac{1}{\lambda}\norm{f}_{L^2(\X,\nu)};\qquad \norm{\D_\alpha u}_{L^2(\X,\nu;H_\alpha)}\leq\sqrt{\frac{2}{\lambda}}\norm{f}_{L^2(\X,\nu)}.
\end{gather*}
Moreover if $N_2$ is symmetric in $L^2(\X,\nu)$, then for every $\lambda>0$ and $f\in L^2(\X,\nu)$, the unique  solution of \eqref{Saiaka} belongs to $W^{2,2}_\alpha(\X,\nu)$ and 
\begin{align*}
\|\D_\alpha^2 u\|_{L^2(\X,\nu;\mathcal{H}_\alpha)}\leq 2\sqrt{2}\norm{f}_{L^2(\X,\nu)},
\end{align*}
where $\mathcal{H}_\alpha$ is the space of Hilbert--Schmidt operators on $H_\alpha$.
\end{thm}
\noindent Since $\xi_A(\X)$ is a core for $N_2$ (Theorem \ref{identif}) then the weak solution $u$ constructed in Theorem \ref{Stime} is strong, namely there exists a sequence $\{u_n\}_{n\in\N}\subseteq \xi_A(\X)$ such that $u_n$ and $\lambda u_n-N_2 u_n$ converge to $u$ and $f$ in $L^2(\X,\nu)$, respectively.

The study of the Sobolev regularity of the solution of stationary second order Kolmogorov equations, like \eqref{Saiaka}, is classical in finite dimension (see, for example, \cite{LAD}), while in infinite dimension the literature on this subject is still scarce. Results like Theorem \ref{Stime} are usually called ``maximal Sobolev regularity results'' since they state the maximal regularity that the weak solution of \eqref{Saiaka} can enjoy in Sobolev spaces. They also provide other information about the domain of $N_2$, in particular Theorem \ref{Stime} states that $\Dom(N_2)$ is continuously embedded in $W^{2,2}_\alpha(\X,\nu)$.

As a final result we will prove the logarithmic Sobolev inequality and the Poincar\'e inequality and some of their consequences in the spaces $W^{1,2}_\alpha(\X,\nu)$. More precisely we will prove the following theorems.

\begin{thm}\label{logsob_pro}
Assume Hypotheses \ref{hyp1} hold true. For $p\geq 1$ and $\varphi\in \mathcal{FC}^1_b(\X)$ (see Section \ref{funspa}), the following inequality holds:
\begin{align}\label{logsob}
\int_\X\abs{\varphi}^p\ln\abs{\varphi}^pd\nu\leq\pa{\int_\X|\varphi|^pd\nu}&\ln\pa{\int_\X|\varphi|^pd\nu}+\frac{p^2}{2\zeta_\alpha}\int_\X\abs{\varphi}^{p-2}\norm{\D_\alpha \varphi}_\alpha^2\chi_{\set{\varphi\neq 0}}d\nu.
\end{align}
Furthermore for every $\varphi\in W^{1,2}_\alpha(\X,\nu)$ it holds
\begin{align}\label{Oldman}
\int_\X\abs{\varphi}^2\ln\abs{\varphi}^2d\nu\leq\pa{\int_\X|\varphi|^2d\nu}&\ln\pa{\int_\X|\varphi|^2d\nu}+\frac{2}{\zeta_\alpha}\int_\X\norm{\D_\alpha \varphi}_\alpha^2\chi_{\set{\varphi\neq 0}}d\nu.
\end{align}
\end{thm}

\begin{thm}\label{Alita}
Assume Hypotheses \ref{hyp1} hold true. If $\varphi\in W_\alpha^{1,2}(\X,\nu)$, then
\begin{gather}\label{poin}
\int_\X\abs{\varphi-\int_\X\varphi d\nu}^2d\nu\leq \frac{1}{2\zeta_\alpha}\int_\X\|\D_\alpha \varphi\|_\alpha^2d\nu.
\end{gather}
\end{thm}

\noindent Such inequalities are present in the literature only in specific settings. In \cite{AN-FE-PA1,CAP-FER1,CAP-FER2} and \cite[Section 12]{DA-ZA1} the authors assume that $F=-\J U$ where $U:\X\ra\R$ is a convex function with Lipschitz continuous derivative and $\alpha=0$ or $\alpha=1/2$. In \cite[Section 3.6]{DA1}, \cite{DA-DE-GO1} and \cite[Section 11]{DA-ZA1} the authors consider a generic Lipschitz continuous function $F$ and $\alpha=0$. In \cite{KA1} the authors assume hypotheses similar to the ones of this papers, but they work in a finite dimensional setting.

The study of Sobolev spaces defined on infinite dimensional spaces, the search for maximal Sobolev regularity results and of logarithmic Sobolev and Poincar\'e inequalities have already been approached by various authors. In \cite[Section 3.6.1]{DA1} and \cite{DA-DE-GO1} the authors assume that $Q^\alpha$ has a continuous inverse and work with the Sobolev space $W^{1,2}(\X,\nu)$ defined as the domain of the closure in $L^2(\X,\nu)$ of the classical Fr\'echet gradient operator $\D:\xi_A(\X)\subseteq L^2(\X,\nu)\ra L^2(\X,\nu;\X)$. We emphasize that the case when $Q^\alpha$ has a continuous inverse presents no significant differences in defining and studying Sobolev spaces compared to the case when $\alpha=0$. In \cite{CAP-FER1} and \cite{FER1} the authors assume that $\alpha={1/2}$, where $Q$ is a positive, self-adjoint and trace class operator, and $F=-Q\D U$ where $U:\X\ra\R$ is a Fr\'echet differentiable and convex function, such that $\D U$ is Lipschitz continuous. They consider the Sobolev space $W^{1,2}_Q(\X,\nu)$ defined as the closure in $L^2(\X,\nu)$ of the operator $Q^{1/2}\D:\xi_A(\X)\subseteq L^2(\X,\nu)\ra L^2(\X,\nu;\X)$. We underline that, if $F=-Q\D U$, then the invariant measure $\nu$ is a weighted Gaussian measures and $N_2$ is the self-adjoint operator associated to the quadratic form 
\[
G(\varphi,\psi)=\int_\X\langle Q^{1/2}\D\varphi,Q^{1/2}\D\psi\rangle d\nu,\qquad \varphi,\psi \in W_Q^{1,2}(\X,\nu).
\] 
Conversely, if $F$ is not of that form, then $N_2$ is not necessarily associated to a quadratic form. In this paper we will revise the methods of the above mentioned papers, to avoid the conditions $Q^{-\alpha}\in\mathcal{L}(\X)$ or $F=-Q\D U$.

The paper is organized in the following way. In Section \ref{pathologic} we fix the notation and recall some basic definitions and results. Section \ref{CLOS} is dedicated to the proof of Theorem \ref{identif}, while Section \ref{supersob} is devoted to the proofs of Theorem \ref{Stime}, Theorem \ref{logsob_pro} and Theorem \ref{Alita}. In the final section we give some examples of operators $A$, $Q$ and of functions $F$ satisfying our hypotheses.

\section{Preliminaries}\label{pathologic}
In this Section we define the notation that we will use in this paper and we recall some basic definitions and results.

\subsection{Notations and remarks}\label{funspa}
Throughout the paper all the integrals are to be understood in the sense of Bochner unless stated otherwise.

Let $H_1$ and $H_2$ be two real Hilbert spaces with inner products $\gen{\cdot,\cdot}_{H_1}$ and $\gen{\cdot,\cdot}_{H_2}$ respectively. We denote by $\mathcal{B}(H_1)$ the family of the Borel subsets of $H_1$ and by $B_b(H_1;H_2)$ the set of the $H_2$-valued, bounded and Borel measurable functions. When $H_2=\R$ we simply write $B_b(H_1)$. 

We denote by $C^k_b(H_1;H_2)$, $k\in\N\cup\set{\infty}$ the set of the $k$-times Fr\'echet differentiable functions from $H_1$ to $H_2$ with bounded derivatives up to order $k$. If $H_2=\R$ we simply write $C_b^k(H_1).$ For a function $\Phi\in C_b^1(H_1;H_2)$ we denote by $\J \Phi(x)$ the Fr\'echet derivative operator of $\Phi$ at the point $x\in H_1$. If $f\in C_b^1(H_1)$, we still denote by $\D f$ its gradient. If $\Phi:H_1\ra H_2$ is Gateaux differentiable we denote by $\J^G\Phi(x)$ the Gateaux derivative operator of $\Phi$ at the point $x\in H_1$. See \cite[Chapter 7]{FAB1}. We recall that assuming Hypothesis \ref{hyp0.5}\eqref{hyp0.5B}, then it holds
\begin{align}\label{ugo}
\gen{\J F(x)h,h}\leq \zeta_2\|h\|^2,\qquad x,h\in\X.
\end{align}

We denote by $\Id_\X\in\mathcal{L}(\X)$ (the set of bounded linear operators from $\X$ to itself) the identity operator on $\X$. We say that $B\in\mathcal{L}(\X)$ is \emph{non-negative} (\emph{positive}) if for every $x\in \X\setminus\set{0}$
\[\gen{Bx,x}\geq 0\ (>0).\]
In an anologous way we define the non-positive (negative) operators. 
We denote by $\mathfrak{X}$ the space of Hilbert--Schmidt operators from $\X$ to $\X$. 
\begin{rmk}
By Hypotheses \ref{hyp0}, there exists an orthonormal basis $\lbrace e_k\rbrace_{k\in \N}$ of $\X$ consisting of eigenvectors of $Q$, i.e.  
\begin{equation*}
Qe_k=\lambda_k e_k,
\end{equation*}
\noindent where $\lambda_k>0$, for any $k\in\N$, are the eigenvalues of $Q$. From here on we fix this orthonormal basis for $\X$. 
\end{rmk}
For any $k\in\N\cup\set{\infty}$, we denote by $\mathcal{FC}^k_b(\X)$ the set of functions $f:\X\ra\R$ such that, for some $n\in\N$, there exists a function $\varphi\in C^k_b(\R^n)$ such that for all $x\in\X$
\[
f(x)=\varphi(\langle x,e_1\rangle,\ldots,\langle x,e_n\rangle).
\]
We call maps of this type cylindrical functions. 

\begin{rmk}
In this paper we have chosen to define the operator $N_0$, introduced in \eqref{OPFO}, on the linear space $\xi_A(\X)$ (that is a subset of the set of cylindrical functions) since it makes some calculations easier, but we could have defined it on $\mathcal{FC}^{\infty}_b(\X)$. Indeed in many papers whose results we will mention the operator $N_0$ is defined on $\mathcal{FC}^{\infty}_b(\X)$ (see, for example, \cite{DA-LU2,DA-LU3}).
\end{rmk}
%

We conclude this section by recalling the following result about the relationship between the uniform convergence on compact sets and the pointwise convergence (see \cite[Definition 43.12, Lemma 43.13 and Theorem 43.14]{WI1}).

\begin{prop}\label{conuni}
A sequence $\{\varphi_n\}_{n\in\N}\subseteq C_{b}(\X)$ is uniformly convergent on every compact subset of $\X$ to a function $\varphi\in C_{b}(\X)$ if, and only if, $\{\varphi_n\}_{n\in\N}$ is pointwise convergent to $\varphi$ and the family  $\{\varphi_n\,|\,n\in\N\}$ is equicontinuos, namely for any $x_0\in\X$ and $\epsilon>0$ there exists $\delta:=\delta(x_0,\eps)>0$ such that, for any $n\in\N$ and $x\in\X$ with $\norm{x-x_0}\leq\delta$ we have $\vert \varphi_n(x)-\varphi_n(x_0)\vert \leq \epsilon$.
\end{prop}

We recall some basic definitions about stochastic process and probability spaces (see \cite{OKS1}). Let $(\Omega,\mathcal{F},\{\mathcal{F}_t\}_{t\geq 0},\mathbb{P})$ be a complete, filtered probability space. Let $\xi:(\Omega,\mathcal{F},\mathbb{P})\ra (\X,\mathcal{B}(\X))$ be a random variable, we denote by 
\[\mathscr{L}(\xi):=\mathbb{P}\circ\xi^{-1}\] 
the law of $\xi$ on $(\X,\mathcal{B}(\X))$, and by
\[\mathbb{E}[\xi]:=\int_\Omega \xi d\mathbb{P}=\int_\X x (\mathscr{L}(\xi))(dx),\]
the expectation of $\xi$ respect to $\mathbb{P}$. 

For the sake of self-completeness, we are going to recall a classical definition reguarding stochastic processes.

\begin{defi}
 We denote by $\X^p([0,T])$, $T>0$, $p\geq 1$, the space of the progressively measurable $\X$-valued processes $\{\psi(t)\}_{t\in[0,T]}$ such that
\[\norm{\psi}^p_{\X^p([0,T])}:=\sup_{t\in [0,T]}\mathbb{E}\big[\norm{\psi(t)}^p\big]<+\infty.\]
\end{defi}
We refer to \cite[Section 4.1.2]{DA-ZA4}, \cite[Section 2.5.1]{LI-RO1} and \cite[Section 1]{PES-ZA1} for the definition of cylindrical Wiener process.

\subsection{Mild solutions and transition semigroups}
In this subsection we study the existence and uniqueness of a mild solution for \eqref{eqFO} and we define the semigroup associated to it.

\begin{defi}
We call mild solution of \eqref{eqFO} a $\X$-valued adapted stochastic process $\set{X(t,x)}_{t\geq 0}$ satisfying 
\begin{equation*}
X(t,x)=e^{tA}x+\int_0^te^{(t-s)A}F(X(s,x))ds+\int_0^te^{(t-s)A}Q^{\alpha}dW(s),\qquad t\geq 0,\ x\in\X;
\end{equation*}
and such that $\mathbb{P}(\int_0^T\norm{X(s,x)}^2ds<+\infty)=1$, for every $x\in\X$ and $T>0$. For simplicity sometimes we denote by $X(t,x)$ the mild solution $\set{X(t,x)}_{t\geq 0}$.
\end{defi}

\begin{prop}\label{contmild}
If Hypotheses \ref{hyp0.5} hold true, then \eqref{eqFO} has a unique mild solution $X(t,x)$. Moreover $X(t,x)$ has a continuous and predictable version and the function $x\mapsto X(t,x)$ is Lipschitz continuous uniformly with respect to $t\geq 0$. In particular the semigroup 
\begin{align}\label{Miwa}
P(t)\varphi(x):=\mathbb{E}[\varphi(X(t,x))],\qquad t\geq 0,\ x\in\X,\ \varphi\in B_b(\X),
\end{align}
is well defined.
\end{prop}
\begin{proof}
The statements follow from some classic results as in \cite[Theorem 5.2.3]{DA-ZA2}, \cite[Theorem 7.13]{DA-ZA4} and \cite[Proposition 2.5]{AD-BA-MA1}. The fact that the function $x\mapsto X(t,x)$ is Lipschitz continuous follows from \cite[Proposition 3.7]{MAS1}. 
\end{proof}




We stress that by \eqref{contconvu} and \cite[Theorems 4.36-5.2-5.11]{DA-ZA4}, the process 
\[W_A(t):=\int^t_0e^{(t-s)A}Q^\alpha ds,\qquad t\geq 0;\] 
has a continuous version and, for any $p\geq 1$,
\begin{equation}\label{cons1}
\sup_{t\geq 0}\E[\norm{W_A(t)}^p]<+\infty.
\end{equation}
By \cite[Theorems 11.33 and 11.34]{DA-ZA4} we have the following result about the existence and uniqueness of the invariant measure for $P(t)$.

\begin{prop}\label{Davied}
Assume Hypotheses \ref{hyp0.5} hold true. The transition semigroup $P(t)$ defined in \eqref{Miwa} has a unique probability invariant measure $\nu$, namely $\nu$ is a probability and Borel measure on $\X$ such that, for any $t\geq 0$ and $\varphi\in B_b(\X)$ we have
\[
\int_{\X}P(t)\varphi d\nu= \int_{\X}\varphi d\nu.
\]
Moreover for any $x\in\X$ and any function $\varphi: \X\ra \R$ bounded and Lipschitz continuous, it holds
\begin{equation}\label{pro4}
\lim_{t\rightarrow+\infty}P(t)\varphi(x)=\int_X\varphi d\nu.
\end{equation}
\end{prop}



We recall a couple of inequalities that follow immediately from the H\"older and Jensen inequalities, and \eqref{Miwa}. For every $\varphi,\psi\in B_{b}(\X)$, $t>0$, $p,q\in [1,\infty]$, such that $1/p+1/q=1$ (with the usual convention that if $p=1$, then $q=\infty$, and viceversa) and $f:\R\ra\R$ a convex function we have
\begin{align}\label{Anya}
|P(t)\varphi\psi| &\leq (P(t)|\varphi|^q)^{1/q}(P(t)|\psi|^p)^{1/p};\\
\label{pro2}
(f\circ P(t))\varphi&\leq P(t)(f\circ\varphi);
\end{align}
By some standard arguments and the invariance of $\nu$, it follows that the transition semigroup is uniquely extendable in $L^p(\X,\nu)$, with $p\geq 1$.
\begin{prop}\label{L22}
Assume Hypothesis \ref{hyp0.5} hold true. The transition semigroup $P(t)$ defined in \eqref{Miwa} is uniquely extendable to a strongly continuous and contraction semigroup $P_p(t)$ in $L^p(\X,\nu)$, for any $p\geq 1$. We will denote by $N_2$ the infinitesimal generator of $P_2(t)$ in $L^2(\X,\nu)$. 
\end{prop}
%


We now recall some results about the differentiability of the mild solution of \eqref{eqFO} with respect to the initial datum. These results follow arguing as in \cite[Section 3]{MAS1}, however, for the reader convenience, we present their proofs in Appendix \ref{MaoMao}.

\begin{prop}\label{derivazmild}
Assume Hypotheses \ref{hyp1} hold true, let $p\geq 2$ and let $X(t,x)$ the mild solution of \eqref{eqFO}. The map $x\mapsto X(\cdot,x)$ is Gateaux differentiable as a function from $\X$ to $\X^{p}([0,T])$ and, for every $x,h\in\X$, the process $\{\J^G X(t,x)h\}_{t\geq 0}$ is the unique mild solution of
\begin{gather*}
\eqsys{
\frac{d}{dt}S_x(t,h)=\big(A+\J F(X(t,x))\big)S_x(t,h), & t>0;\\ 
S_x(0,h)=h.
}
\end{gather*}
\end{prop}

\begin{prop}\label{stisti}
In the hypotheses of Theorem \ref{derivazmild}, for any $t>0$ and $x,h\in\X$, it holds 
\begin{gather*}
\|\J^G X(t,x)h\|\leq e^{-\zeta t}\norm{h},
\end{gather*}
where $X(t,x)$ is the mild solution of \eqref{eqFO} and $\zeta$ is the constant appearing in Hypotheses \ref{hyp0.5}.
\end{prop}

\subsection{The Ornstein--Uhlenbeck semigroup}
In this subsection we recall some results about the behavior on spaces of continuous functions of the Ornstein--Uhlenbeck semigroup. These preliminary results will be essential for the proof of Theorem \ref{identif}. The Ornstein--Uhlenbeck semigroup in various settings is one of the most studied semigroup in the literature. We refer to \cite{AD21,ASvN13,BO95,FO-ME21,HINO08,LUNMEPA20,LUNPA20,MAvN08,MEPA02,PRI03,vN15} and the reference therein for an in-depth study of this semigroup and of its infinitesimal generator in various spaces.

Consider the Banach space
\[C_{b,2}(\X):=\set{f:\X\ra \R \tc x\mapsto\frac{f(x)}{1+\norm{x}^2}\text{ belongs to } C_b(\X)}.\] 
endowed with the norm
\[\norm{f}_{b,2}:=\sup_{x\in\X}\left(\frac{\vert f(x)\vert}{1+\norm{x}^2}\right),\qquad f\in C_{b,2}(\X).\]
We consider the Ornstein--Uhlenbeck semigroup defined, for $t>0$, $x\in\X$ and $f\in B_b(\X)$, as
\begin{equation}\label{Koyuki}
T(t)f(x):=\int f(e^{tA}x+y)\mathcal{N}(0,Q_t)(dy),
\end{equation}
\noindent where we recall that $Q_tx:=\int^t_0e^{sA}Q^{2\alpha}e^{sA}xds$ and $\mathcal{N}(0,Q_t)$ is the Gaussian measure with mean zero and covariance operator $Q_t$.
\begin{rmk}
We stress that the Ornstein--Uhlenbeck semigroup is the transition semigroup associated to \eqref{eqFO} when $F \equiv 0$, namely
\begin{gather}\label{eq0}
\eqsys{
dZ(t,x)=AZ(t,x)dt+Q^{\alpha}dW(t), & t>0;\\
Z(0,x)=x\in \X.
}
\end{gather}
\noindent Indeed, let $Z(t,x)=e^{tA}x+W_A(t)$ be the mild solution of \eqref{eq0}. Since $W_A(t)\sim\mathcal{N}(0,Q_t)$, then $Z(t,x)$ is a translation of a Gaussian variable. So, via a change of variable, for any $f\in B_b(\X)$, we obtain
\begin{equation*}
\E[f(Z(t,x))]=\int_{\Omega}f(e^{tA}x+W_A(t))d\mathbb{P}=\int_\X f(y)\mathcal{N}(e^{tA}x, Q_t)(dy)=T(t)f(x).
\end{equation*}
\end{rmk}
For a detailed study of the semigroup $T(t)$, defined in \eqref{Koyuki}, in spaces of continuous functions with weighted sup-norms, we refer to \cite{CER2, CER3}, \cite[Section 2.8.3]{DA1} and \cite[Section 2]{DA-TU2}. We are more interested in its behaviour with respect to the mixed topology. For an in-depth study of the mixed topology we refer to \cite{GO-KO1}, in the following theorems we list some properties that will be useful to our aims.

\begin{thm}[Proposition 2.3 and Theorem 4.1 of \cite{GO-KO1}]\label{KOKO}
$ $
\begin{enumerate}[\rm (i)]
\item A sequence $\{\phi_n\}_{n\in\N}\subseteq C_{b,2}(\X)$ converges with respect to the mixed topology to $\phi\in C_{b,2}(\X)$ if, and only if,
\[
\sup_{n\in\N}\norm{\phi_n}_{b,2}<+\infty,
\]
and, for any $K\subseteq\X$ compact, 
\[
\lim_{n\rightarrow +\infty}\sup_{x\in K}\left(\vert \phi_n(x)-\phi(x)\vert\right)=0.
\]

\item The semigroup $T(t)$, introduced in \eqref{Koyuki}, is strongly continuous on $C_{b,2}(\X)$ with respect to the mixed topology.
\end{enumerate}
\end{thm}

\begin{defi}\label{KOKOII}
We denote by $L_{b,2}$ the infinitesimal generator, with respect to the mixed topology, of the semigroup $T(t)$ in $C_{b,2}(\X)$.
\end{defi}

\begin{prop}[Theorems 4.2 and 4.5 of \cite{GO-KO1}]\label{KOKO1}
$ $
\begin{enumerate}[\rm (i)]
\item For any $\lambda>0$, $\varphi\in C_{b,2}(\X)$ and $x\in\X$, the Riemann integral
\begin{equation*}
J(\lambda)\varphi:=\int^{+\infty}_0e^{-\lambda t}T(t)\varphi dt,
\end{equation*}
is well defined and convergent with respect to the mixed topology. Moreover, for every $\lambda>0$, the operator
\[J(\lambda):(C_{b,2}(\X),\tau_M)\ra(C_{b,2}(\X),\tau_M)\]
is continuous (here $\tau_M$ denotes the mixed topology), and $J(\lambda)\varphi=R(\lambda,L_{b,2})\varphi$.

\item  $L_{b,2}$ is an extension to $C_{b,2}(\X)$, endowed with the mixed topology, of the operator $L_0$ defined as 
\begin{align*}
L_0\varphi(x):=\frac{1}{2}\tr[Q^{2\alpha}\D^2\varphi(x)]+\scal{x}{A\D\varphi(x)},\qquad x\in\X, \varphi\in \xi_A(\X).
\end{align*}
\end{enumerate}
\end{prop}

\noindent We define the domain of $L_{b,2}$ with respect to the mixed topology as
\[
\Dom(L_{b,2}):=\set{\varphi\in C_{b,2}(\X)\tc \lim_{t\rightarrow 0}\frac{T(t)\varphi-\varphi}{t}\mbox{ exists with respect to the mixed topology}}.
\]

\noindent By Theorem \ref{KOKO}, we obtain the following charaterization of $\Dom(L_{b,2})$.
\begin{prop}\label{pisemi}
A function $\varphi\in C_{b,2}(\X)$ belongs to $\Dom(L_{b,2})$ if, and only if, there exists $\psi\in C_{b,2}(\X)$ such that
\begin{enumerate}[\rm (i)]
\item for any compact subset $K$ of $\X$, 
\[
\lim_{t\rightarrow 0}\sq{\sup_{x\in K}\pa{\frac{T(t)\varphi(x)-\varphi(x)}{t}-\psi(x)}}=0;
\]
\item $\sup_{t\in(0,1]}[t^{-1}\norm{T(t)\varphi-\varphi}_{b,2}]<+\infty.$
\end{enumerate}
In this case $L_{b,2}\varphi=\psi$.
\end{prop}
We remark that, by \cite[Remark 4.3]{GO-KO1}, Theorem \ref{KOKO} and Theorem \ref{KOKO1}(i), the operator $L_{b,2}$ is the weak infinitesimal generator of the semigroup $T(t)$ on $C_{b,2}(\X)$ in the sense of \cite{CER2, CER3}. By this fact we can use the following approximation result.
\begin{prop}[Propositions 2.5 and 2.6 of \cite{DA-TU2}]\label{appxiA}
Let $\varphi\in \Dom(L_{b,2})\cap C^1_b(\X)$. There exists a family $\{\varphi_{n_1,n_2,n_3,n_4}\,|\,n_1,n_2,n_3,n_4\in\N\}\subseteq\xi_A(\X)$ such that for every $x\in\X$
\begin{align*}
\lim_{n_1\rightarrow+\infty}\lim_{n_2\rightarrow+\infty}\lim_{n_3\rightarrow+\infty}\lim_{n_4\rightarrow+\infty}&\varphi_{n_1,n_2,n_3,n_4}(x)=\varphi(x);\\
\lim_{n_1\rightarrow+\infty}\lim_{n_2\rightarrow+\infty}\lim_{n_3\rightarrow+\infty}\lim_{n_4\rightarrow+\infty}&\D\varphi_{n_1,n_2,n_3,n_4}(x)=\D\varphi(x);\\ 
\lim_{n_1\rightarrow+\infty}\lim_{n_2\rightarrow+\infty}\lim_{n_3\rightarrow+\infty}\lim_{n_4\rightarrow+\infty}&L_{b,2}\varphi_{n_1,n_2,n_3,n_4}(x)=L_{b,2}\varphi(x).
\end{align*}
Furthermore there exists a positive constant $C_\varphi$ such that, for any ${n_1,n_2,n_3,n_4}\in\N$ and $x\in\X$, it holds
\begin{equation}\label{Konatsu}
\vert\varphi_{n_1,n_2,n_3,n_4}(x)\vert+ \norm{\D\varphi_{n_1,n_2,n_3,n_4}(x)}+\vert L_{b,2}\varphi_{n_1,n_2,n_3,n_4}(x)\vert\leq C_{\varphi}(1+\norm{x}^2).
\end{equation}
\end{prop}
\noindent For a proof of the previous result we refer to \cite[Section 2.8.3]{DA1} or \cite[Section 2]{DA-TU2}. Moreover, we underline that the family $\{\varphi_{n_1,n_2,n_3,n_4}\,|\,n_1,n_2,n_3,n_4\in\N\}$ is formally defined in \cite[Section 8]{DA-LU-TU1}.

\section{Closure of $N_0$ and a core for $N_2$}\label{CLOS}

In this Section we will show that the infinitesimal generator $N_2$ of $P_2(t)$ (see Proposition \ref{L22}) is the closure in $L^2(\X,\nu)$ of the operator $N_0$ defined in \eqref{OPFO}. In particular $\xi_A(\X)$ is a core for $P_2(t)$. Before we do that in Subsection \ref{momenu} we will prove that the invariant measure $\nu$ has finite moments of any order and in Subsection \ref{reF} we will introduce a regularizing sequence for $F$.

\subsection{Moments of the invariant measure $\nu$}\label{momenu}

In this subsection we will show that $\nu$ has finite moments (Theorem \ref{mome}). First of all we are going to prove that the mild solution $X(t,x)$ of \eqref{eqFO} has finite moments of every order, for any $t>0$ and $x\in\X$. We start by quickly recalling the following inequality, which is an easy consequence of the Jensen inequality: for every $h_1,h_2\in\X$ and $r\geq 1$ it holds
\begin{align}\label{Nina}
\norm{h_1-h_2}^r\geq 2^{1-r}\norm{h_1}^r-\norm{h_2}^r.
\end{align}

\begin{prop}\label{momemild}
Assume Hypotheses \ref{hyp0.5} hold true. For every $k\geq 2$, $t>0$ and $x\in\X$, there exist two positive constants $\beta:=\beta(k)$ and $\gamma:=\gamma(k,x,m)$, such that for every $t>0$ and $x\in\X$
\[\E[\norm{X(t,x)}^k]:=\int_\X\norm{y}^k\nu_{t,x}(dy)<\gamma(1+\norm{x}^ke^{-\beta t}),\]
where $\nu_{t,x}:=\mathscr{L}(X(t,x))$ and $m$ is the constant appearing in Hypothesis \ref{hyp0.5}\eqref{hyp0.5A}. In particular the mild solution $X(t,x)$ belongs to $\X^p([0,T])$, for any $p\geq 2$.
\end{prop}

\begin{proof}
We prove the statement with $k=2n$ and $n\in\N$. For every $x\in\X$ we consider the process
\[Y(t,x)=X(t,x)-W_A(t),\quad t>0.\]
Observe that $Y(t,x)$ is the mild solution of the following stochastic partial differential equation
\begin{gather}\label{eqff}
\eqsys{
\frac{d}{dt}Y(t,x)=AY(t,x)+F(Y(t,x)+W_A(t)), & t>0;\\
Y(0,x)=x\in \X.
}
\end{gather}
Throughout the proof we will assume that $Y(t,x)$ is a strict solution of \eqref{eqff}, the general case follows by an approximation method (see proof of \cite[Theorem 5.5.8]{DA-ZA2}). 
Scalarly multiplying both members of the first equation in \eqref{eqff} by $\norm{Y(t,x)}^{2n-2}Y(t,x)$, we have
\begin{align*}
\|Y&(t,x)\|^{2n-2}\scal{\frac{d}{dt}Y(t,x)}{Y(t,x)}\\
&=\norm{Y(t,x)}^{2n-2}\scal{AY(t,x)}{Y(t,x)}+\norm{Y(t,x)}^{2n-2}\scal{F(Y(t,x)+W_A(t))}{Y(t,x)}\\
& =\norm{Y(t,x)}^{2n-2}\scal{AY(t,x)}{Y(t,x)}+\norm{Y(t,x)}^{2n-2}\scal{F(Y(t,x)+W_A(t))-F(W_A(t))}{Y(t,x)}\\
&\phantom{000000000000}+\norm{Y(t,x)}^{2n-2}\scal{F(W_A(t))}{Y(t,x)}.
\end{align*}
and so, by Hypotheses \ref{hyp0.5}, we have
\begin{equation}\label{Ysti1}
\frac{1}{2n}\dfrac{d}{dt}\norm{Y(t,x)}^{2n}\leq -\zeta\norm{Y(t,x)}^{2n}+\norm{F(W_A(t))}\norm{Y(t,x)}^{2n-1}.
\end{equation}
We claim that there exist two positive constants $C_1$ and $C_2$, depending only on $n$ and on $\zeta$, such that
\begin{equation}\label{Ysti}
\frac{1}{2n}\dfrac{d}{dt}\norm{Y(t,x)}^{2n}\leq -C_1\norm{Y(t,x)}^{2n}+\frac{C_2}{2n}\norm{F(W_A(t))}^{2n}.
\end{equation}
To prove \eqref{Ysti} we have to consider the two cases $\zeta\in (0,1]$ and $\zeta>1$ separately. Let $\zeta\in (0,1]$, by the Young inequality we have
\begin{align*}
\norm{F(W_A(t))}\norm{Y(t,x)}^{2n-1}&=\left(\frac{1}{\zeta}\norm{F(W_A(t))}\right)\left(\zeta\norm{Y(t,x)}^{2n-1}\right)\\
&\leq \frac{2n-1}{2n}\zeta^{\frac{2n}{2n-1}}\norm{Y(t,x)}^{2n}+\frac{\zeta^{-2n}}{2n}\norm{F(W_A(t))}^{2n}
\end{align*}
and from \eqref{Ysti1} we get
\[
\frac{1}{2n}\dfrac{d}{dt}\norm{Y(t,x)}^{2n}\leq \left(\frac{2n-1}{2n}\zeta^{\frac{2n}{2n-1}}-\zeta\right)\norm{Y(t,x)}^{2n}+\frac{\zeta^{-2n}}{2n}\norm{F(W_A(t))}^{2n}.
\]
We stress that $\frac{2n-1}{2n}\zeta^{\frac{2n}{2n-1}}-\zeta<0$ for $n\in\N$ and $\zeta\in(0,1]$ and so \eqref{Ysti} is verified. Let $\zeta>1$, by the Young inequality we have
\begin{align*}
\norm{F(W_A(t))}\norm{Y(t,x)}^{2n-1}&=\left(\zeta\norm{F(W_A(t))}\right)\left(\frac{1}{\zeta}\norm{Y(t,x)}^{2n-1}\right)\\
&\leq \frac{2n-1}{2n}\zeta^{-\frac{2n}{2n-1}}\norm{Y(t,x)}^{2n}+\frac{\zeta^{2n}}{2n}\norm{F(W_A(t))}^{2n}
\end{align*}
and from \eqref{Ysti1} we get
\[
\frac{1}{2n}\dfrac{d}{dt}\norm{Y(t,x)}^{2n}\leq \left(\frac{2n-1}{2n}\zeta^{-\frac{2n}{2n-1}}-\zeta\right)\norm{Y(t,x)}^{2n}+\frac{\zeta^{2n}}{2n}\norm{F(W_A(t))}^{2n}.
\]
Observe that $\frac{2n-1}{2n}\zeta^{-\frac{2n}{2n-1}}-\zeta<0$ for $n\in\N$ and $\zeta>1$ and so \eqref{Ysti} is verified. Taking the expectations in \eqref{Ysti} we obtain
\[
\dfrac{d}{dt}\E[\norm{Y(t,x)}^{2n}]\leq -2nC_1\E[\norm{Y(t,x)}^{2n}]+C_2\E[\norm{F(W_A(t))}^{2n}],
\]
and, by the variation of constants formula, we have
\[
\E[\norm{Y(t,x)}^{2n}]\leq \norm{x}^{2n}e^{-2nC_1 t}+C_2\int^t_0e^{-2nC_1(t-s)}\E[\norm{F(W_A(s))}^{2n}]ds.
\]
\noindent By \eqref{cons1} and Hypothesis \ref{hyp0.5}\eqref{hyp0.5A}, for any $s>0$, we have
\begin{align*}
\E[\norm{F(W_A(s))}^{2n}]&\leq C^{2n}\E\left[(1+\norm{W_A(s)}^m)^{2n}\right]\\
&=C^{2n}\sum^{2n}_{i=0}\binom{2n}{i}\E[\norm{W_A(s)}^{im}]=:C_3<+\infty.
\end{align*}
 So 
\begin{align*}
\E[\norm{Y(t,x)}^{2n}]&=\E[\norm{X(t,x)-W_A(t)}^{2n}]\\
&\leq \norm{x}^{2n}e^{-2nC_1 t}+C_2C_3\int^t_0e^{-2nC_1(t-s)}ds\leq \norm{x}^{2n}e^{-2nC_1 t}+\frac{C_2C_3}{2nC_1} .
\end{align*}
\noindent By \eqref{Nina} we have that 
\begin{equation*}
\norm{X(t,x)-W_A(t)}^{2n}\geq \frac{1}{2^{2n-1}}\norm{X(t,x)}^{2n}-\norm{W_A(t)}^{2n}.
\end{equation*}
\noindent So, by \eqref{cons1}, we obtain
\[
\E[\norm{X(t,x)}^{2n}]\leq 2^{2n-1}\norm{x}^{2n}e^{-2nC_1 t}+\frac{2^{2n-2}C_2C_3}{nC_1} +2^{2n-1}sup_{t\geq 0}E[\norm{W_A(t)}^{2n}]<+\infty,
\]
\noindent so the thesis is verified with $\beta=2nC_1$ and $\gamma=\max\{2^{2n-1}, \frac{2^{2n-2}C_2C_3}{nC_1} +2^{2n-1}c_{2n}(\Tr[Q])^{2n}\}$. 
To obtain the thesis for generic $k\geq 2$, it is sufficient to note that there exists $n\in\N$ such that $2(n-1)\leq k\leq 2n$, and so for any $x\in\X$, we have
\(
\norm{x}^{k}\leq \max\{\norm{x}^{2(n-1)},\norm{x}^{2n}\}.
\)
\end{proof}
 
By \eqref{pro4} we know that for any bounded and Lipschitz continuous function $\varphi:\X\ra\R$ and any $x\in\X$, we have
\begin{align}\label{Kaguya}
\lim_{t\ra+\infty}P(t)\varphi(x)&=\lim_{t\rightarrow+\infty}\E[\varphi(X(t,x))]=\lim_{t\rightarrow+\infty}\int_\X\varphi(y)(\mathscr{L}(X(t,x)))(dy)=\int_\X\varphi(y)\nu(dy).
\end{align}
Hence to estimate the moment of $\nu$ of order $k\geq 2$, we can exploit Proposition \ref{momemild}.

\begin{thm}\label{mome}
If Hypotheses \ref{hyp0.5} hold true and $k\geq 1$, then $\int_\X\norm{y}^{k}\nu(dy)< +\infty$.
\end{thm}

\begin{proof}
Let $k\geq 2$. For any $t>0$ and $x\in\X$, we set \(\nu_{t,x}:=\mathscr{L}(X(t,x))\). Let $b>0$, we have
\[
\int_\X \dfrac{\norm{y}^{k}}{1+b\norm{y}^{k}}\nu_{t,x}(dy)\leq \int_\X \norm{y}^{k}\nu_{t,x}(dy)=\mathbb{E}[\norm{X(t,x)}^{k}],
\]
\noindent by Proposition \ref{momemild}, we know
\begin{equation*}
\lim_{t\rightarrow+\infty}\mathbb{E}[\norm{X(t,x)}^{k}]\leq \lim_{t\rightarrow+\infty}\gamma(1+e^{-\beta t}\norm{x}^k)=\gamma<+\infty,
\end{equation*}
\noindent then, by \eqref{pro4}, \eqref{Kaguya} and the monotone convergence theorem, we conclude
\[
\int_\X\norm{y}^{2p}\nu(dy)=\lim_{b\rightarrow 0}\lim_{t\rightarrow+\infty}\int_\X \dfrac{\norm{y}^{2p}}{1+b\norm{y}^{2p}}\nu_{t,x}(dy)<+\infty.
\]
The case $k\in [1,2)$ follows by the H\"older inequality.
\end{proof}

As an immediate consequence of the previous theorem and the polynomial growth of $F$ (Hypotheses \ref{hyp0.5}), we have that
\begin{equation}\label{FNL2}
\int_\X \norm{F}^2 d\nu<+\infty.
\end{equation}

\subsection{A regularizing family for F}\label{reF}

\noindent This section is dedicated to the introduction of a family of Lipschitz continuous and smooth functions $\{F_{\delta,s}\,|\,\delta,s>0\}$ approximating $F$, such that for every $\delta,s>0$, the function $F_{\delta,s}$ is dissipative with the same constant of $F$. First of all we consider the Yosida approximants of $F$ (see \cite[Section 7.1]{BRE1}). For any $\delta>0$ and $x\in\X$ we set
\[
F_\delta(x):=F(x_\delta),
\]
\noindent where $x_\delta$ is the unique solution of the equation
\begin{equation}\label{eq-yos}
y-\delta (F(y)-\zeta_2 y)=x.
\end{equation}
\noindent The function $G(x):=F(x)-\zeta_2 x$ is $m$-dissipative, so by  \cite[Proposition 5.5.3]{DA-ZA2}, for any $\delta>0$, \eqref{eq-yos} has a unique solution and
\begin{align}
&\norm{G(x_{\delta})-G(z_{\delta})}\leq \frac{2}{\delta}\norm{x-z}\label{1yy};\\
&\scal{G(x_{\delta})-G(z_{\delta})}{x-z}\leq 0\label{2yy};\\
&\norm{G(x_{\delta})}\leq \norm{G(x)}\label{3yy}.
\end{align}
By \eqref{1yy} it follows immediately that $F_\delta$ is Lipschitz continuous and by \eqref{eq-yos} and \eqref{2yy} we have 
\begin{align*}
0\geq& \gen{G(x_\delta)-G(z_\delta),x-z}=\gen{F(x_\delta)-F(z_\delta),x-z}-\zeta_2\gen{x_\delta-z_\delta,x-z}\\
=& \gen{F(x_\delta)-F(z_\delta),x-z}+\zeta_2\gen{\delta(F(z_\delta)-\zeta_2 z_\delta)-\delta(F(x_\delta)-\zeta_2 x_\delta)+z-x,x-z}\\
=& \gen{F(x_\delta)-F(z_\delta),x-z}+\zeta_2\delta\gen{F(z_\delta)-F(x_\delta),x-z}\\
&\phantom{aaaaaaaaaaaa00000000000000000000}+\zeta_2^2\delta \gen{x_\delta-z_\delta,x-z}+\zeta_2 \gen{z-x,x-z}\\
=& \gen{F(x_\delta)-F(z_\delta),x-z}-\zeta_2\norm{x-z}^2+\zeta_2\delta\gen{G(z_\delta)-G(x_\delta),x-z}\\
\geq & \scal{F_{\delta}(x)-F_{\delta}(z)}{x-z}-\zeta_2\norm{x-z}^2,
\end{align*}
and so \(\scal{F_{\delta}(x)-F_{\delta}(z)}{x-z}\leq \zeta_2\norm{x-z}^2\). Moreover, for any $\delta>0$ and $x\in\X$, by \eqref{3yy} we have
\begin{align*}
\norm{F_\delta(x)-F(x)}\leq & \norm{G(x_\delta)-G(x)}+\zeta_2\norm{x_\delta-x}\notag\\
\leq &\norm{G(x_\delta)-G(x)}+\zeta_2\delta\norm{G(x_\delta)} \leq (2+\zeta_2\delta)\norm{G(x)}.
\end{align*}

\noindent We need to introduce a further regularization. Let $B\in \mathcal{L}(\X)$ be a positive and trace class operator. For every $\delta,s>0$ and $x,k\in\X$, we define 
\[
\scal{F_{\delta,s}(x)}{k}:=\int_\X \langle F_{\delta}(e^{-(s/2)B^{-1}}x+y),k\rangle\mathcal{N}(0,B_s)(dy),
\]
\noindent where we recall that $B_s=B(\Id-e^{-sB^{-1}})$. This type of regularization is classical and it is based on the Mehler formula for the Ornstein--Uhlenbeck semigroup (see \cite[Proof of Theorem 11.2.14]{DA-ZA1}). For any $s>0$ and $x,z\in\X$, we have that $F_{\delta,s}(x)$ is Lipschitz continuous and
\begin{align}
\scal{F_{\delta,s}(x)-F_{\delta,s}(z)}{x-z}\leq \zeta_2\norm{x-z}^2;\label{122}
\end{align}
For any $\delta>0$, $F_\delta$ is Lipschitz continuous so
\begin{equation*}
\sup_{s\geq 0}\sup_{x\in\X}\norm{F_{\delta,s}(x)}<+\infty.
\end{equation*}
\noindent We stress that, for any $s>0$, we have
\begin{align}
e^{-(s/2)B^{-1}}(\X) \subseteq B_s^{1/2}(\X).\label{Fii}
\end{align}
Indeed by the analyticity of $e^{-(s/2)B^{-1}}$ and by \cite[Proposition 2.1.1(i)]{LUN1}  the range of $e^{-(s/2)Q^{-1}}$ is contained in the domain of $B^{-k}$ for every $k\in\N$. So to prove \eqref{Fii}  it is sufficient to prove that $\Id-e^{-sB^{-1}}$ is invertible. Since $-B^{-1}$ is negative, we have $\|e^{-sQ^{-1}}\|_{\mathcal{L}(\X)}<1$, and so $\Id-e^{-sB^{-1}}$ is invertible. In particular $B_s^{1/2}(\X)= B^{1/2}(\X)$ and so we get \eqref{Fii}. Moreover \eqref{Fii} combined with the Cameron--Martin formula implies that $F_{\delta,s}$ is Gateaux differentiable (see for example \cite[p. 259]{DA-ZA1}) and by \eqref{122}, for any $x,y\in\X$, we have
\[
\langle\J^G F_{\delta,s}(x)y,y\rangle\leq \zeta_2\norm{y}^2.
\]
We conclude this subsection by recalling the following result, which is a easy consequence of \eqref{FNL2} and the properties stated above.
\begin{prop}\label{555}
If Hypotheses \ref{hyp0.5} hold true, then \(
\lim_{\delta\rightarrow 0}\lim_{s\rightarrow 0}F_{\delta,s}=F
\), where the two limits are taken in $L^2(X,\nu)$.
\end{prop}

\subsection{Proof of Theorem \ref{identif}}
The first step to prove Theorem \ref{identif} is to show that $N_2$ is an extension of $N_0$ to $L^2(\X,\nu)$. Let $\varphi\in\xi_A(\X)$, then there exist $m,n\in\N$, $a_1,\ldots,a_m,b_1,\ldots,b_n\in \R$ and $h_1,\ldots,h_m,k_1,\ldots,k_n\in \Dom(A) $ such that
\[\varphi(x)=\sum_{i=1}^m a_i\sin(\gen{x,h_i})+\sum_{j=1}^n b_j\cos(\gen{x,k_j}).\]
Easy computations give for $x\in\X$
\begin{align*}
N_0\varphi(x)=\sum_{i=1}^m a_i&\pa{\gen{x,Ah_i}+\gen{F(x),h_i}-\frac{1}{2}\|Q^{\alpha}h_i\|}\sin(\gen{x,h_i})\notag\\
&\phantom{0000000}+\sum_{j=1}^n b_j\pa{\gen{x,Ak_j}+\gen{F(x),k_j}-\frac{1}{2}\|Q^{\alpha}k_j\|}\cos(\gen{x,k_j}),
\end{align*}
and observe that, by Hypotheses \ref{hyp0.5} and Theorem \ref{mome}, $N_0\varphi$ belongs to $L^2(\X,\nu)$

\begin{prop}
Assume Hypotheses \ref{hyp0.5} hold true. If $\varphi\in\xi_A(\X)$, then $\varphi\in\Dom(N_2)$, and $N_2\varphi=N_0\varphi$. Moreover $N_2$ is an extension of $\overline{N}_0$ in $L^2(\X,\nu)$, where $\overline{N}_0$ is the closure of $N_0$ in $L^2(\X,\nu)$.
\end{prop}
\begin{proof}
Let $\varphi\in\xi_A(\X)$, by \cite[Proposition 3.19]{DA1}, for any $x\in\X$, it holds
\[
P(t)\varphi(x)=\E[\varphi(X(t,x))]=\varphi(x)+\E\left[\int_0^tN_0\varphi(X(s,x))ds\right],
\]
\noindent and so we obtain $\lim_{t\rightarrow 0}t^{-1}(P(t)\varphi(x)-\varphi(x))=N_0\varphi(x)$.
\noindent We want to apply the Vitali convergence theorem (see \cite[Theorem 2.24]{FO-LE1}), so we need to show that the family $\{t^{-1}(P(t)\varphi-\varphi)\,|\, t>0\}$ is $2$-equi-integrable. Fix $\eps>0$ and consider $\delta>0$ such that whenever $E\in\mathcal{B}(\X)$ with $\nu(E)<\delta$, then
\[\int_E|N_0\varphi(x)|^2\nu(dx)<\eps.\]
Now let $E\in\mathcal{B}(\X)$ such that $\nu(E)<\delta$, then by the Jensen inequality, \eqref{pro2} and the invariance of $P(t)$ with respect to $\nu$ we have
\begin{align*}
\frac{1}{t^2}\int_E\vert P(t)\varphi(x)-\varphi(x))\vert^2  \nu(dx)&=\int_E\abs{\int_0^tP(s)N_0\varphi(x)\frac{ds}{t}}^2\nu (dx)\\
&\leq \frac{1}{t}\int_0^t\left(\int_E\abs{P(s)N_0\varphi(x)}^2\nu (dx)\right)ds\\
&\leq  \frac{1}{t}\int_0^t\left(\int_E P(s)\vert N_0\varphi(x)\vert^2\nu (dx)\right)ds\\
&=  \frac{1}{t}\int_0^t\left(\int_E\vert N_0\varphi(x)\vert^2\nu (dx)\right)ds\\
&<  \frac{1}{t}\int_0^tP(s)\eps ds=\frac{1}{t}\int_0^t\eps ds=\eps.
\end{align*}
\noindent so, by the Vitali convergence theorem, we obtain $t^{-1}(P(t)\varphi-\varphi)$ converges to $N_0\varphi$ in $L^2(\X,\nu)$, as $t$ goes to zero. 
The closability of $N_0$ in $L^2(\X,\nu)$ follows by its dissipativity, and we denote by $\overline{N}_0$ its closure in $L^2(\X,\nu)$. By the facts that $N_2$ is closed in $L^2(\X,\nu)$ and that $\xi_A(\X)\subseteq \Dom(N_2)$, it follows that $N_2$ is an extension of $\overline{N}_0$.
\end{proof} 
\begin{rmk}
In a similar way we can prove an analogous result for the operator $N_0$ defined on the space of cylindrical functions $\mathcal{FC}^k_b(\X)$, with $k\in\N\cup\set{\infty}$.
\end{rmk}


Before proving Theorem \ref{identif} we need two preliminary results.

\begin{lemm}\label{inc}
Assume Hypotheses \ref{hyp0.5} hold true. Let $\varphi\in\Dom(L_{b,2})\cap C^1_b(\X)$, then $\varphi\in\Dom(\overline{N}_0)$ and
\begin{equation*}
\overline{N}_0\varphi(x)=L_{b,2}\varphi(x)+\scal{F(x)}{\D\varphi(x)},\qquad x\in\X;
\end{equation*}
where $L_{b,2}$ is the operator introduced in Definition \ref{KOKOII}.
\end{lemm}
\begin{proof}
By Proposition \ref{appxiA} a family $\{\varphi_{n_1,n_2,n_3,n_4}\,|\, n_1,n_2,n_3,n_4\in\N\}\subseteq\xi_A(\X)$ exists such that, for any $x\in\X$,
\[
\lim_{n_1\rightarrow+\infty}\lim_{n_2\rightarrow+\infty}\lim_{n_3\rightarrow+\infty}\lim_{n_4\rightarrow+\infty}\overline{N}_0\varphi_{n_1,n_2,n_3,n_4}(x)=L_{b,2}\varphi(x)+\scal{F(x)}{\D\varphi(x)}.
\]
whenever $\varphi\in\Dom(L_{b,2})\cap C^1_b(\X)$. So, by Hypothesis \ref{hyp0.5}\eqref{hyp0.5A}, \eqref{OPFO} and \eqref{Konatsu}, there exists $K_\varphi>0$, such that for every $x\in\X$
\[
\vert \overline{N}_0\varphi_{n_1,n_2,n_3,n_4}(x)\vert=\vert N_0\varphi_{n_1,n_2,n_3,n_4}(x)\vert\leq K_\varphi(1+\norm{x}^{m+2}).
\]
So we can conclude the proof using the dominated convergence theorem and Theorem \ref{mome}.
\end{proof}

The following proposition is a general result about closed operators. Since we were unable to find an appropriate reference in the literature we provide its proof.

\begin{prop}\label{cloop}
Let $Y$ be a Banach space and let $B_1:\Dom(B_1)\subseteq Y\ra Y$ and $B_2:\Dom(B_2)\subseteq Y\ra Y$ be two, possibly unbounded, operators. If
\begin{enumerate}[\rm (i)]
\item $B_1$ is an extension of $B_2$, namely $\Dom(B_2)\subseteq \Dom(B_1)$ and, for any $x\in\Dom(B_2)$, it holds $B_2x=B_1x$;
\item there exists a dense subset $D$ of $Y$ such that, for some $\lambda>0$, $R(\lambda, B_1)$ and  $R(\lambda, B_2)$ are well defined, and $R(\lambda,B_1)(D)\subseteq \Dom(B_2)$;
\end{enumerate}
then $\Dom(B_1)=\Dom(B_2)$ and $B_1=B_2$. 
\end{prop}

\begin{proof}
For any $x\in D$ 
\[
x=(\Id_Y\lambda-B_1)R(\lambda,B_1)x=\lambda R(\lambda,B_1)x-B_1R(\lambda,B_1)x.
\]
By the fact that $R(\lambda,B_1)(D)\subseteq\Dom(B_2)$ and that $B_1$ is an extension of $B_2$, it follows
\[
x=\lambda R(\lambda,B_1)x-B_2R(\lambda,B_1)x=(\Id_Y\lambda-B_2)R(\lambda,B_1)x,
\]
hence, for any $x\in D$, we have \(R(\lambda,B_2)x=R(\lambda,B_1)x\). So by the density of $D$ in $Y$, for any $x\in Y$, we have shown that \(R(\lambda,B_2)x=R(\lambda,B_1)x\). Recalling that the domain of an operator coincides with the range of its resolvent, we get the thesis.
\end{proof}

\noindent Finally we can prove Theorem \ref{identif}.
\begin{proof}[Proof of Theorem \ref{identif}]
By Proposition \ref{cloop} to prove Theorem \ref{identif} it is sufficient to show that there exists a dense subset $D$ of $L^2(\X,\nu)$ such that
\begin{equation*}
R(\lambda, N_2)(D)\subseteq \Dom(\overline{N}_0).
\end{equation*}
\noindent We split the proof in two steps. In the first step we assume that $F$ is Gateaux differentiable and Lipschitz continuous, and we will show that we can take $C^1_b(\X)$ as the set $D$. In the second step we will show that, in the general case, the set $(\lambda\Id_{\X}-N_2)(\Dom(\overline{N}_0))$ is dense in $L^2(\X,\nu)$ and it can be chosen as the set $D$. Throughout the proof we let $X(t,x)$ be the mild solution of \eqref{eqFO}.

\noindent\textbf{\emph{Step 1.}} Assume that $F$ is Gateaux differentiable and Lipschitz continuous. Let $f\in C^1_b(\X)$ and $\lambda>0$, consider the function $\varphi$ defined as 
\[
\varphi(x):=R(\lambda, N_2)f(x)=\int^{+\infty}_0e^{-\lambda s}P(s)f(x)ds,\qquad x\in\X.
\]
We want to show that $\varphi$ is Gateaux differentiable and to do so we will use the dominated convergence theorem. By Proposition \ref{stisti} and \cite[Corollary 21]{BF20} we have
\begin{align*}
\lim_{\delta\ra 0}\frac{1}{\delta}\sq{f(X(t,x+\delta h))-f(X(t,x))}=\gen{\D f(X(t,x)),\J^GX(t,x)h}.
\end{align*}
Furthermore, by Proposition \ref{stisti}, it holds
\begin{align*}
\frac{1}{\delta}|f(X(t,x+\delta h))-f(X(t,x))|&=\frac{1}{\delta}\abs{\int_0^\delta\gen{\D f(X(t,x+sh)),\J^GX(t,x+sh)h}ds}\\
&\leq e^{-\zeta t}\norm{\D f}_\infty\norm{h}.
\end{align*}
So by the dominated convergence theorem we have that $\varphi$ is Gateaux differentiable. Now we want to show that $\varphi$ is actually Fr\'echet differentiable. To do so we will show an uniform estimate for $\J^G\varphi$ and we will use \cite[Fact 1.13(b), p. 8]{PH1}.
\begin{align*}
|\J^G\varphi(x)h|&=\lim_{\delta\ra 0}\frac{1}{\delta}|\varphi(x+\delta h)-\varphi(x)|=\lim_{\delta\ra 0}\frac{1}{\delta}\abs{\int_0^{+\infty}e^{-\lambda s}(P(s)f(x+\delta h)-P(s)f(x))ds}\notag\\
&\leq\lim_{\delta\ra 0}\frac{1}{\delta}\int_0^{+\infty}e^{-\lambda s}\E\sq{\abs{f(X(s,x+\delta h))-f(X(s,x))}}ds\notag\\
&=\lim_{\delta\ra 0}\frac{1}{\delta}\int_0^{+\infty}e^{-\lambda s}\E\sq{\abs{\int_0^\delta\gen{\D f(X(s,x+rh)),\J^GX(s,x+rh)h}dr}}ds\notag\\
&\leq \norm{\D f}_\infty\norm{h}\int_0^{+\infty}e^{-(\lambda+\zeta)s}ds=\frac{1}{\lambda+\zeta}\norm{\D f}_\infty\norm{h}.
\end{align*}
Since the above estimate is uniform with respect to the elements of $\X$ of norm one, by \cite[Fact 1.13(b), p. 8]{PH1} the function $\varphi$ is Fr\'echet differentiable, and 
\begin{equation}\label{uuu}
\norm{\D \varphi}_{\infty}\leq \dfrac{1}{\lambda+\zeta}\norm{\D f}_{\infty}.
\end{equation} 
Finally we have to prove that $\varphi$ belongs to $\Dom(L_{b,2})$. Let $Z(t,x)$ be the mild solution of \eqref{eq0}, we have
\[
Z(t,x)=X(t,x)-\int_0^t e^{(t-s)A}F(X(s,x))ds.
\]
Then, for every $x\in\X$, we have 
\begin{align}\label{Kase-san}
\frac{T(t)\varphi(x)-\varphi(x)}{t}&=\frac{\E[\varphi(Z(t,x))-\varphi(x)]}{t}\notag\\
&=\frac{1}{t}\E\sq{\varphi\pa{X(t,x)-\int_0^t e^{(t-s)A}F(X(s,x))ds}-\varphi(x)}.
\end{align}
Using the Taylor formula in the last term of \eqref{Kase-san}, we obtain
\begin{align*}
\frac{T(t)\varphi(x)-\varphi(x)}{t}=&\frac{P(t)\varphi(x)-\varphi(x)}{t}
-\frac{1}{t}\E\left[\scal{\D \varphi(X(t,x))}{\int_0^t e^{(t-s)A}F(X(s,x))ds}\right]\\
&\phantom{00000111111111111111111}+o\pa{\E\sq{\norm{\int_0^t e^{(t-s)A}F(X(s,x))ds}}}.
\end{align*}
Now let $K$ be a compact subset of $\X$ and observe that, by Proposition \ref{momemild}, we get for $x\in K$
\[o\pa{\E\sq{\norm{\int_0^t e^{(t-s)A}F(X(s,x))ds}}}=o(t),\qquad t\ra 0,\]
where $o(t)$ does not depend on the choice of $x$ in $K$. Hence, for any $x\in K$, we have
\begin{equation}\label{covpunt}
L_{b,2}\varphi(x)=\lim_{t\rightarrow 0}\dfrac{1}{t}(T(t)\varphi(x)-\varphi(x))=N_2\varphi(x)-\scal{\D \varphi(x)}{F(x)}.
\end{equation}
We are going to check the conditions of Proposition \ref{pisemi} to obtain that $\varphi$ belongs to $\Dom(L_{b,2})$. We begin to check (i) of Proposition \ref{pisemi}. We set for $t>0$ and $x\in\X$  
\[
\Delta_t(x):=\frac{P(t)\varphi(x)-\varphi(x)}{t},\quad R_t(x):=\frac{1}{t}\E\left[\scal{\D \varphi(X(t,x))}{\int_0^t e^{(t-s)A}F(X(s,x))ds}\right].
\]
By Proposition \ref{conuni} and \eqref{covpunt}, it is sufficient to prove that the family $\{\Delta_t-R_t\,|\,t\in[0,T]\}$ is equicontinuos. We recall that for every $t\geq 0$
\[
P(t)\varphi=P(t)\int^{+\infty}_0e^{-\lambda s}P(s)fds=e^{\lambda t}\int^{+\infty}_t e^{-\lambda s}P(s)fds.
\]
We now show that $\{\Delta_t\,|\,t\in[0,T]\}$ is an equicontinuous family. Let $x_0\in K$. By the continuity of $f$ we know that for every $\eps>0$ there exists $\delta>0$ such that $|f(x)-f(x_0)|\leq \eps$, whenever $\norm{x-x_0}\leq \delta$. Now let $\norm{x-x_0}\leq\delta$
\begin{align*}
|\Delta_t(x)-\Delta_t(x_0)|&=\frac{1}{t}\abs{P(t)\varphi(x)-\varphi(x)-P(t)\varphi(x_0)+\varphi(x_0)}\\
&=\frac{1}{t}\abs{e^{-\lambda t}\int_t^{+\infty}e^{-\lambda s}P(s)(f(x)-f(x_0))ds+\int_0^{+\infty}e^{-\lambda s}P(s)(f(x_0)-f(x))ds}\\
&=\frac{1}{t}\Bigg|(e^{-\lambda t}-1)\int_t^{+\infty}e^{-\lambda s}P(s)(f(x)-f(x_0))ds\\
&\phantom{000000000000000000000000000000}+\int_0^{t}e^{-\lambda s}P(s)(f(x_0)-f(x))ds\Bigg|\\
&\leq\frac{e^{-\lambda t}-1}{t}\int_t^{+\infty}e^{-\lambda s}P(s)\eps ds+\frac{1}{t}\int_0^{t}e^{-\lambda s}P(s)\eps ds\\
&\leq\eps\pa{\frac{e^{-\lambda t}-1}{t}\int_t^{+\infty}e^{-\lambda s} ds+\frac{1}{t}\int_0^{t}e^{-\lambda s} ds}\\
&=\eps e^{-\lambda t}\frac{1-e^{-\lambda t}}{\lambda t}\leq \eps,
\end{align*}
so the family $\{\Delta_t\,|\,t\in[0,T]\}$ is equicontinuous. Now we study the equicontinuity of $\{R_t\,|\,\in[0,T]\}$. We start by observing that by the Lipschitz continuity of $F$ there exists $C>0$ such that for every $x\in\X$, it holds $\norm{F(x)}\leq C(1+\norm{x})$. Furthermore by \eqref{contmild} and \eqref{uuu}, for every $t>0$, the functions $x\mapsto \D\varphi(X(t,x))$ and $x\mapsto F(X(t,x))$ are continuous uniformly with respect to $t\in[0,T]$. So for every $t\in[0,T]$, $x_0\in K$ and $\eps>0$ there exists $\delta:=\delta(\eps,x_0)>0$ such that whenever $\norm{x-x_0}\leq \delta$ it holds
\begin{align*}
\max\set{\norm{\D\varphi(X(t,x))-\D\varphi(X(t,x_0))},\norm{F(X(t,x))-F(X(t,x_0))}}\leq \eps.
\end{align*}
By the Jensen inequality and Proposition \ref{momemild} we can write
\begin{align*}
|R_t(x)-R_t(x_0)|&=\frac{1}{t}\Bigg|\E\sq{\gen{\D\varphi(X(t,x)),\int_0^te^{(t-s)A}F(X(s,x))ds}}\\
&\phantom{11111111111111111111}-\E\sq{\gen{\D\varphi(X(t,x_0)),\int_0^te^{(t-s)A}F(X(s,x_0))ds}}\Bigg|\\
&=\frac{1}{t}\Bigg|\E\sq{\gen{\D\varphi(X(t,x))-\D\varphi(X(t,x)),\int_0^te^{(t-s)A}F(X(s,x))ds}}\\
&\phantom{1111111111}+\E\sq{\gen{\D\varphi(X(t,x_0)),\int_0^te^{(t-s)A}(F(X(s,x))-F(X(s,x_0)))ds}}\Bigg|\\
&\leq\frac{1}{t}\E\sq{\norm{\D\varphi(X(t,x))-\D\varphi(X(t,x))}\norm{\int_0^te^{(t-s)A}F(X(s,x))ds}}\\
&\phantom{1111111111}+\frac{1}{t}\E\sq{\norm{\D\varphi(X(t,x_0))}\norm{\int_0^te^{(t-s)A}(F(X(s,x))-F(X(s,x_0)))ds}}\\
&\leq\frac{C\eps}{t}\int_0^t\E\sq{1+\norm{X(s,x)}}ds+\frac{\norm{\D\varphi}_\infty}{t}\E\sq{\int_0^t\norm{F(X(s,x))-F(X(s,x_0))}ds}\\
&\leq \eps\pa{C(1+2\gamma)+\norm{\D\varphi}_\infty},
\end{align*}
where $\gamma$ is the constant appearing in Proposition \ref{momemild}. So the family $\{\Delta_t-R_t\,|\,t\in [0,T]\}$ is equicontinuos. Using similar arguments also condition (ii) of Proposition \ref{pisemi} is verified since $\varphi\in C^1_b(\X)$ and we have assumed $F$ to be  Lipschitz continuous in this first step. Consequently $\varphi\in \Dom(L_{b,2})$, in particular
\[
L_{b,2}\varphi=N_2\varphi-\scal{\D \varphi}{F},
\]
and
\[
\lambda\varphi-L_{b,2}\varphi-\scal{\D \varphi}{F}=f.
\]
In the end we have proved that $\varphi\in \Dom(L_{b,2})\cap C^1_b(\X)$ and, by Lemma \ref{inc}, we conclude that $\varphi\in \Dom(\overline{N}_0)$ and 
\[
\overline{N}_0\varphi=L_{b,2}\varphi+\scal{\D \varphi}{F}.
\]
\noindent\textbf{\emph{Step 2.}} Let $\{F_{\delta,s}\,|\,\delta,s>0\}$ be the regularizing family of $F$ defined in Section \ref{reF}. Let $f\in C^1_b(\X)$, for any $\delta,s>0$, we set
\[
\varphi_{\delta,s}(x):=\int^{+\infty}_0e^{-\lambda t}P_{\delta,s}(t)f(x)dt,\quad x\in\X,
\]
where $P_{\delta,s}(t)$ is the transition semigroup associated to
\begin{gather*}
\eqsys{
dX(t,x)=\big(AX(t,x)+F_{\delta,s}(X(t,x))\big)dt+Q^{\alpha}dW(t), & t>0;\\
X(0,x)=x\in \X.
}
\end{gather*}
In Section \ref{reF} we have seen that, for any $\delta,s>0$, the function $F_{\delta,s}$ is Lipschitz continuous and verifies Hypotheses \ref{hyp0.5} with the constant $\zeta_2$. Hence by Step 1, for any $\delta,s>0$, we have that $\varphi_{\delta,s}\in\Dom(\overline{N}_0)$ and 
\[
\lambda\varphi_{\delta,s}-L_{b,2}\varphi_{\delta,s}-\scal{\D \varphi_{\delta,s}}{F_{\delta,s}}=f.
\]
So 
\[
\lambda\varphi_{\delta,s}-\overline{N}_0\varphi_{\delta,s}=f+\scal{\D\varphi_{\delta,s}}{F_{\delta,s}-F},
\]
and recalling that $N_2$ is an extension of $\overline{N}_0$ in $L^2(X,\nu)$
\[
\lambda\varphi_{\delta,s}-N_2\varphi_{\delta,s}=f+\scal{\D \varphi_{\delta,s}}{F_{\delta,s}-F},
\]
where the equality holds in $L^2(\X,\nu)$. Hence noticing that \eqref{uuu} does not depend on $\delta$ and on $s$ and by Remark \ref{555}, we have that, for any $f\in C^1_b(\X)$, there exist a family $\{\varphi_{\delta,s}\,|\,\delta,s>0\}\subseteq\Dom(\overline{N}_0)$, such that 
\begin{equation*}
\lim_{\delta\rightarrow 0}\lim_{s\rightarrow 0}(\lambda\Id_{\X}-N_2)\varphi_{\delta,s}=f,\quad \mbox{ in } L^2(X,\nu).
\end{equation*}
By the density of $C^1_b(\X)$ in $L^2(\X,\nu)$ we get the density of $(\lambda\Id_\X-N_2)(\Dom(\overline{N}_0))$ in $L^2(\X,\nu)$.
\end{proof}
\begin{rmk}
At the cost of complicating the calculations it is possible to prove the same result in $L^p(\X,\nu)$, for any $p\geq 1$.
\end{rmk}

\section{Sobolev spaces}\label{supersob}

In this section we are going to analyze the behaviour of the operator $N_2$ in $L^2(\X,\nu)$. We will concetrate in particular on the domain of $N_2$ and we will prove some useful inequalities. 

\subsection{$H_\alpha$ differentiability} 
In this first subsection we state and proof some auxiliary results that we will use in the rest of the paper. We start with the following proposition.

\begin{prop}
Assume Hypotheses \ref{hyp1} hold true. For every $t>0$, the map $x\mapsto X(t,x)$ is $\mathbb{P}$-a.e. $H_\alpha$-Gateaux differentiable and for any $x\in\X$ and $h\in H_\alpha$ its Gateaux derivative along $h\in  H_\alpha$ is $\{\J^G X(t,x)h\}_{t\geq 0}$.
\end{prop}

\begin{proof}
By Hypotheses \ref{hyp1} and standard calculations we have that $\J F=Q^{2\alpha}\J G$. By Proposition \ref{derivazmild}, the map $x\mapsto X(\cdot,x)$ is Gateaux differentiable as a function from $\X$ to $\X^{p}([0,T])$ and, for every $x,h\in\X$, the process $\{\J^G X(t,x)h\}_{t\geq 0}$ is the unique mild solution of
\begin{gather*}
\eqsys{
\frac{d}{dt}S_x(t,h)=\big(A+\J F(X(t,x))\big)S_x(t,h), & t>0;\\ 
S_x(0,h)=h,
}
\end{gather*}
namely
\begin{equation}\label{mildder}
\J^G X(t,x)h=e^{tA}h+\int^t_0 e^{(t-s)A}F(X(s,x))\J^G X(s,x)hds.
\end{equation}
By Hypotheses \ref{hyp1} and \eqref{mildder}, for any $x\in\X$ and $h\in H_\alpha$, the process $\{\J^G X(t,x)h\}_{t\geq 0}$ has trajectories in $H_\alpha$. Moreover for any $x\in\X$, $h\in H_\alpha$ and $r,t>0$ we have
\begin{align*}
&\E\left[\norm{\dfrac{X(t,x+rh)-X(t,x)}{r}-\J^G X(t,x)h}_\alpha\right]\\
&\leq\norm{Q^{\alpha}}_{\mathcal{L}(\X)}\int^t_0 E\left[\norm{\dfrac{G(X(s,x+rh))-G(X(s,x))}{r}-\J G(X(s,x))\J^G X(s,x)h}\right]ds,
\end{align*}
by the dominated convergence theorem we have
\[ 
\lim_{r\ra 0}\E\left[\norm{\dfrac{X(t,x+rh)-X(t,x)}{r}-\J^G X(t,x)h}_\alpha\right]=0.
\]
This conclude the proof.
\end{proof}

Using Hypotheses \ref{hyp1} as in the proof of \eqref{stisti}, for any $t>0$, $x\in\X$ and $h\in H_\alpha$ we obtain
\begin{align*}
\norm{\J^G X(t,x)h}\leq e^{-\zeta_\alpha t}\norm{h}_\alpha
\end{align*}
Moreover, by \cite[Corollary 21]{BF20}, we have
\(
\D_\alpha P_2(t)\varphi(x)=\E\sq{(\J^G X(t,x))^*\D_\alpha\varphi(X(t,x))}.
\)
So the following result follows.

\begin{lemm}
For every $\varphi\in \mathcal{FC}_b^1(\X)$ it holds
\begin{align}\label{Azala}
\|\D_\alpha P_2(t)\varphi\|_\alpha^2\leq e^{-2\zeta_\alpha t}P_2(t)\|Q^{2\alpha}\D\varphi\|^2.
\end{align}
\end{lemm}
Finally we recall the following proposition (see \cite[Proposition 17]{BF20}).

\begin{prop}
Let $Y$ be a Hilbert space and let $\Phi: \X\rightarrow Y$. If $\Phi$ is Fr\'echet differentiable at $x\in\X$, then it is differentiable along $H_\alpha$ and, for every $h\in H_\alpha$,
\begin{align*}
\J\Phi(x)h=\J\Phi(x)h.
\end{align*}
Furthermore if $\varphi:\X\ra\R$ is Fr\'echet differentiable at $x\in\X$, then $\D_\alpha\varphi(x)=Q^{2\alpha}\D\varphi(x)$.
\end{prop}

We remark that the set $\{h_i=\lambda^\alpha e_i\,|\,i\in\N\}$ provides an orthonormal basis for $H_\alpha$. From here on we fix such orthonormal basis. Throughout this section we will denote by $\mathcal{H}_\alpha$ the space of the Hilbert--Schmidt operators from $H_\alpha$ to itself.
\subsection{Closability of $\D_\alpha$}\label{bloom}

In this section we introduce the Sobolev spaces we will use throughout the rest of the paper. In order to do so we need a couple of lemmata.

\begin{lemm}
Assume Hypotheses \ref{hyp0.5} hold true and let $\varphi,\psi\in \xi_A(\X)$. Then the product $\varphi\psi$ belongs to $\xi_A(\X)$ and 
\begin{align}
N_2(\varphi\psi) &=\varphi N_2\psi+\psi N_2\varphi +\langle Q^{\alpha}\D\varphi,Q^{\alpha}\D\psi\rangle=\varphi N_2\psi+\psi N_2\varphi +\langle \D_\alpha\varphi,\D_\alpha\psi\rangle_\alpha.\label{N_2quadro}
\end{align}
Furthermore whenever $\varphi\in \Dom(N_2)$ and $g\in C_b^2(\R)$, we have
\begin{align}\label{Hazel}
\int_\X(g'\circ \varphi)N_2\varphi d\nu=-\int_\X(g''\circ\varphi)\|Q^{\alpha}\D\varphi\|^2d\nu.
\end{align}
\end{lemm}

\begin{proof}
The fact that $\varphi\psi$ belongs to $\xi_A(\X)$ and \eqref{N_2quadro} follow by direct calculations. We recall that $N_2 u=N_0 u$ whenever $u\in \xi_A(\X)$ (Theorem \ref{identif}). Now we prove \eqref{Hazel}. We start by showing that when $\psi$ belongs to $\Dom(N_2)$ it holds
\begin{align}\label{intparts}
\int_\X \psi N_2\psi d\nu=-\frac{1}{2}\int_\X \|\D_\alpha\psi\|_\alpha^2d\nu.
\end{align}
To do so it is enough to recall that $\nu$ is invariant. Indeed by \eqref{N_2quadro} we have for $\varphi\in\xi_A(\X)$
\begin{align*}
0=\int_\X N_2\varphi^2 d\nu=\int_\X\pa{2\varphi N_2\varphi +\|\D_\alpha\varphi\|_\alpha^2}d\nu.
\end{align*}
Since $\xi_A(\X)$ is a core for $N_2$, by \eqref{intparts} and the Young inequality, it follows that
\[\D_\alpha: \xi_A(\X)\subseteq \Dom(N_2)\ra L^2(\X,\nu;H_\alpha),\qquad  \varphi\mapsto \D_\alpha\varphi,\]
is continuous and, consequently, it can be extended to all of $\Dom(N_2)$ (endowed with the graph norm). We shall still denote by $\D_\alpha$ its extension. So \eqref{intparts} follows by a standard density argument. In a similar way we can use the dominated convergence theorem to get \eqref{Hazel}.
\end{proof}

The next result is a technical lemma about the behaviour of the derivative of the semigroup $P_2(t)$ which will be useful to prove the closability of the gradient operator (Proposition \ref{Algernon}) and the Poincar\'e inequality (Proposition \ref{Alita}).

\begin{lemm}
Assume Hypotheses \ref{hyp1} hold true and let $\varphi\in \xi_A(\X)$. It holds
\begin{align}\label{3dinotte}
\int_\X|P_2(t)\varphi|^2d\nu+\int_0^t\int_\X\|\D_\alpha P_2(s)\varphi\|_\alpha^2d\nu ds=\int_\X|\varphi|^2d\nu.
\end{align}
\end{lemm}

\begin{proof}
For every $\varphi\in \xi_A(\X)$ we have
\begin{align}\label{caldissimo}
\frac{d}{ds}P_2(s)\varphi=N_2P_2(s)\varphi,\qquad s>0.
\end{align}
Multiplying both sides of \eqref{caldissimo} by $P_2(s)\varphi$, integrating on $\X$ with respect to $\nu$, and taking into account \eqref{intparts}, we find
\begin{align}\label{Carnevale}
\int_\X\frac{d}{ds}|P_2(s)\varphi|^2 d\nu=-\int_\X\|\D_\alpha P_2(s)\varphi\|_\alpha^2d\nu.
\end{align}
Now the thesis follows integrating \eqref{Carnevale} with respect to $s$ from $0$ to $t$.
\end{proof}

We are now in the right position to prove the closability of the derivative operator $\D_\alpha$ and $\D_\alpha^2$ that we will use to define the Sobolev spaces $W_\alpha^{1,2}(\X,\nu)$ and $W_\alpha^{2,2}(\X,\nu)$.

\begin{prop}\label{Algernon}
Assume Hypotheses \ref{hyp1} hold true. The operators $\D_\alpha:\xi_A(\X)\subseteq L^2(\X,\nu)\ra L^2(\X,\nu;H_\alpha)$ and $(\D_\alpha, \D_\alpha^2):\xi_A(\X)\subseteq L^2(\X,\nu)\ra L^2(\X,\nu;H_\alpha)\times L^2(\X,\nu;\mathcal{H}_\alpha)$ are closable.
\end{prop}

\begin{proof}
We assume that $\{\varphi_n\}_{n\in\N}\subseteq \xi_A(\X)$ is a sequence such that
\begin{align}
L^2(\X,\nu)&\text{-}\lim_{n\ra+\infty}\varphi_n =0;\label{Hamilton}\\
L^2(\X,\nu;H_\alpha)&\text{-}\lim_{n\ra+\infty}\D_\alpha\varphi_n =\Psi.\notag
\end{align}
By \eqref{3dinotte}, the strong continuity of $P_2(t)$ and \eqref{Hamilton}, we have
\begin{align}\label{Orpheus}
\lim_{n\ra+\infty}\int_0^t\int_\X\|\D_\alpha P_2(s)\varphi_n\|_\alpha^2d\nu ds=\lim_{n\ra+\infty}\pa{\int_\X|\varphi_n|^2d\nu-\int_\X|P_2(t)\varphi_n|^2d\nu}=0
\end{align}
We also claim that 
\begin{align}\label{Euridice}
\lim_{n\ra+\infty}\int_0^t\int_\X\|\D_\alpha P_2(s)\varphi_n\|_\alpha^2d\nu ds=\int_0^t\int_\X\norm{\E\sq{(\J^G X(t,x))^*\Psi(X(t,x))}}_\alpha^2\nu(dx)ds.
\end{align}
Indeed by \cite[Corollary 21]{BF20} we have
\begin{align*}
\D_\alpha P_2(t)\varphi_n(x)=\E\sq{(\J^G X(t,x))^*\D_\alpha\varphi_n(X(t,x))}
\end{align*}
Observe that
\begin{align*}
&\int_0^t\int_\X\norm{\E\sq{(\J^G X(s,x))^*\D_\alpha\varphi_n(X(s,x))}-\E\sq{(\J^G X(t,x))^*\Psi(X(s,x))}}_\alpha^2\nu(dx)ds\\
\leq & \int_0^t\int_\X e^{-2\zeta_\alpha s}\norm{\E\sq{\D_\alpha\varphi_n(X(s,x))-\Psi(X(s,x))}}_\alpha^2\nu(dx)ds\\
\leq & \int_0^t\int_\X e^{-2\zeta_\alpha s}P_2(s)\norm{{\D_\alpha\varphi_n(x)-\Psi(x)}}_\alpha^2\nu(dx)ds.
\end{align*}
Recalling that $\nu$ is invariant for $P_2(t)$ we have
\begin{align*}
0\leq &\lim_{n\ra+\infty}\int_0^t\int_\X\norm{\E\sq{(\J^G X(s,x))^*\D_\alpha\varphi_n(X(s,x))}-\E\sq{(\J^G X(s,x))^*\Psi(X(s,x))}}_\alpha^2\nu(dx)ds\\
\leq & \lim_{n\ra+\infty}\int_0^t\int_\X e^{-2\zeta_\alpha s}\norm{{\D_\alpha\varphi_n(x)-\Psi(x)}}_\alpha^2\nu(dx)ds=0.
\end{align*}
This prove \eqref{Euridice}. Combining \eqref{Orpheus} and \eqref{Euridice} we get 
\begin{align*}
\int_0^t\int_\X\norm{\E\sq{(\J^G X(s,x))^*\Psi(X(s,x))}}_\alpha^2\nu(dx)ds=0.
\end{align*}
So for a.e. $s\in(0,t)$ (with respect to the Lebesgue measure) it holds
\begin{align}\label{Persefone}
\int_\X\norm{\E\sq{(\J^G X(s,x))^*\Psi(X(s,x))}}_\alpha^2\nu(dx)=0.
\end{align}
To be more precise we denote by $A$ the subset, with measure zero, of $(0,t)$, such that in $(0,t)\setminus A$ \eqref{Persefone} does not hold. For $s\in A$, by the monotone convergence theorem, we have
\begin{align*}
0=&\int_\X\norm{\E\sq{(\J^G X(s,x))^*\Psi(X(s,x))}}_\alpha^2\nu(dx)\\
=&\int_\X\sum_{i=1}^{+\infty}\abs{\gen{\E\sq{(\J^G X(s,x))^*\Psi(X(s,x))},h_i}_\alpha}^2\nu(dx)\\
=&\sum_{i=1}^{+\infty}\int_\X\abs{\E\sq{\gen{(\J^G X(s,x))^*\Psi(X(s,x)),h_i}_\alpha}}^2\nu(dx)\\
=&\sum_{i=1}^{+\infty}\int_\X\abs{\E\sq{\gen{\Psi(X(s,x)),\J^G X(s,x)h_i}_\alpha}}^2\nu(dx).
\end{align*}
So for $s\in A$ and $i\in\N$
\[\int_\X\abs{\E\sq{\gen{\Psi(X(s,x)),\J^G X(s,x)h_i}_\alpha}}^2\nu(dx)=0.\]
Now observe that for $s\in A$ and $i\in\N$ we have
\begin{align*}
0\leq &\norm{P_2(s)(\gen{\Psi(\cdot),h_i}_\alpha)}_{L^2(\X,\nu)}=\norm{\E\sq{\gen{\Psi(X(s,\cdot)),h_i}_\alpha}}_{L^2(\X,\nu)}\\
=&\norm{\E\sq{\gen{\Psi(X(s,\cdot)),h_i}_\alpha}}_{L^2(\X,\nu)}-\norm{\E\sq{\gen{\Psi(X(s,\cdot)),\J^G X(s,\cdot)h_i}_\alpha}}_{L^2(\X,\nu)}\\
\leq & \norm{\E\sq{\gen{\Psi(X(s,\cdot)),h_i}_\alpha}-\E\sq{\gen{\Psi(X(s,\cdot)),\J^G X(s,\cdot)h_i}_\alpha}}_{L^2(\X,\nu)}\\
=&\norm{\E\sq{\gen{\Psi(X(s,\cdot)),h_i-\J^G X(s,\cdot)h_i}_\alpha}}_{L^2(\X,\nu)},
\end{align*}
by the continuity of $s\mapsto\J^G X(s,\cdot)$ and the dominated convergence theorem we get that for every $i\in\N$
\[\|\gen{\Psi(\cdot),h_i}_\alpha\|_{L^2(\X,\nu)}=0.\]
By a standard argument we get that $\Psi(x)=0$ for $\nu$-a.e $x\in X$. This prove the closability of $\D_\alpha:\xi_A(\X)\subseteq L^2(\X,\nu)\ra L^2(\X,\nu;H_\alpha)$. The closability of $(\D_\alpha, \D_\alpha^2):\xi_A(\X)\subseteq L^2(\X,\nu)\ra L^2(\X,\nu;H_\alpha)\times L^2(\X,\nu;\mathcal{H}_\alpha)$ follows by similar calculations substituting $\varphi_n$ with $\gen{D_\alpha\varphi_n,h_i}_\alpha$ for $i\in\N$. 
\end{proof}

We are now able to define the Sobolev spaces we will use throughout the rest of the paper.
\begin{defi}
We define the Sobolev spaces $W_\alpha^{1,2}(\X,\nu)$ and $W_\alpha^{2,2}(\X,\nu)$ as the domains of the closure of the operators $\D_\alpha:\xi_A(\X)\subseteq L^2(\X,\nu)\ra L^2(\X,\nu;H_\alpha)$ and $(\D_\alpha, \D_\alpha^2):\xi_A(\X)\subseteq L^2(\X,\nu)\ra L^2(\X,\nu;H_\alpha)\times L^2(\X,\nu;\mathcal{H}_\alpha)$, respectively.
\end{defi}

\begin{rmk}
It is known that $\mu\sim\mathcal{N}(0,Q)$ is the unique invariant measure of the Ornstein--Uhlenbeck semigroup $T(t)$ introduced in \eqref{Koyuki} (for any $\alpha\geq 0$). An interesting question is when the measure $\nu$ is absolutely continuous with respect to the measure $\mu$. When $\alpha=0$ this problem has already been addressed in \cite[Section 3.7]{DA1}, where it is proved that if $F$ is a bounded and Fr\'echet differentiable function with bounded derivative operator, then $\nu\ll\mu$. The case $\alpha=1/2$ has been studied in \cite[Theorem 3.5]{BO-RO1}, with the additional hypotheses that $F(\X)\subseteq H_{1/2}$ and $F\in L^2(\X,\mu; H_{1/2})$. We plan to study this problem in our setting in an upcoming work.
\end{rmk}

We are now able to prove Theorem \ref{Stime}.

\begin{proof}[Proof of Theorem \ref{Stime}]
Since $\xi_A(\X)$ is a core for $N_2$, then a sequence $\set{u_n}_{n\in\N}\subseteq \xi_A(\X)$ exists such that $u_n$ converges to a function $u$ in $L^2(\X,\nu)$ and
\[L^2(\X,\nu)\text{-}\lim_{n\ra+\infty}\lambda u_n-N_2 u_n=f.\]
Let $f_n:=\lambda u_n-N_2 u_n$. Multiplying $f_n$ by $u_n$, integrating with respect to $\nu$ and using \eqref{Hazel} we get
\begin{gather*}
\int_\X f_n u_n d\nu=\lambda \int_\X u_n^2 d\nu-\int_\X u_n N_2 u_n d\nu=\lambda \int_\X u_n^2 d\nu+\frac{1}{2}\int_\X \norm{\D_\alpha u_n}_\alpha^2 d\nu.
\end{gather*}
By the Cauchy--Schwarz inequality we get
\begin{gather*}
\norm{u_n}_{L^2(\X,\nu)}\leq\frac{1}{\lambda}\norm{f_n}_{L^2(\X,\nu)};\qquad \norm{\D_\alpha u_n}_{L^2(\X,\nu;H_\alpha)}\leq\sqrt{\frac{2}{\lambda}}\norm{f_n}_{L^2(\X,\nu)}.
\end{gather*}
Since $\set{u_n}_{n\in\N}$ and $\set{f_n}_{n\in\N}$ converge to $u$ and $f$, respectively, in $L^2(\X,\nu)$ we get
\[\norm{u}_{L^2(\X,\nu)}=\lim_{n\ra+\infty}\norm{u_n}_{L^2(\X,\nu)}\leq\lim_{n\ra+\infty}\frac{1}{\lambda} \norm{f_n}_{L^2(\X,\nu)}=\frac{1}{\lambda}\norm{f}_{L^2(\X,\nu)}.\]
Moreover
\[\norm{\D_\alpha u_n-\D_\alpha u_m}_{L^2(\X,\nu;H_\alpha)}\leq\sqrt{\frac{2}{\lambda}}\norm{f_n-f_m}_{L^2(\X,\nu)},\]
then $\{\D_\alpha u_n\}_{n\in\N}$ is a Cauchy sequence in $L^2(\X,\nu;H_\alpha)$. By the closability of $\D_\alpha$ in $L^2(\X,\nu)$ (Proposition \ref{Algernon}) it follows that $u\in W_\alpha^{1,2}(\X,\nu)$ and
\[L^2(\X,\nu;H_\alpha)\text{-}\lim_{n\ra+\infty}\D_\alpha u_n=D_\alpha u.\]
Therefore
\[\norm{\D_\alpha u}_{L^2(\X,\nu;H_\alpha)}=\lim_{n\ra+\infty}\norm{\D_\alpha u_n}_{L^2(\X,\nu;H_\alpha)}\leq\lim_{n\ra+\infty}\sqrt{\frac{2}{\lambda}} \norm{f_n}_{L^2(\X,\nu)}=\sqrt{\frac{2}{\lambda}}\norm{f}_{L^2(\X,\nu)}.\]

Now we prove the moreover part of the statement. Using \eqref{OPFO}, we differentiate the equality $\lambda u_n-N_2 u_n=f_n$ along $h_j$ direction, we multiply the result by $\langle \D_\alpha u, h_j\rangle_\alpha$, sum over $j$ and finally integrate over $\X$ with respect to $\nu$. We obtain 
\begin{align*}
\lambda\int_\X\norm{\D_\alpha u_n}_\alpha^2d\nu&-\int_\X\langle \D_\alpha u_n,A\D_\alpha u_n\rangle_\alpha d\nu+\frac{1}{2}\int_\X\|\D_\alpha^2 u_n\|_{\mathcal{H}_\alpha}^2d\nu\\
&-\int_\X\langle Q^{2\alpha}\D G \D_\alpha u_n,\D_\alpha u_n\rangle_\alpha d\nu=\int_\X\langle \D_\alpha f_n,\D_\alpha u_n\rangle_\alpha d\nu.
\end{align*}
Recalling that $\langle (A+\J F(x))h,h\rangle_\alpha \leq -\zeta\norm{h}_\alpha^2$ for every $x\in \X$ and $h\in H_\alpha$ we have
\begin{align*}
(\lambda+\zeta_\alpha)\int_\X\norm{\D_\alpha u_n}_\alpha^2d\nu&+\frac{1}{2}\int_\X\|\D_\alpha^2 u_n\|_{\mathcal{H}_\alpha}^2d\nu\leq \int_\X\langle \D_\alpha f_n,\D_\alpha u_n\rangle_\alpha d\nu.
\end{align*}
Finally we have
\begin{align}\label{Maiko}
\frac{1}{2}\int_\X\|\D_\alpha^2 u_n\|_{\mathcal{H}_\alpha}^2d\nu\leq \int_\X\langle \D_\alpha f_n,\D_\alpha u_n\rangle_\alpha d\nu.
\end{align}
Now from \eqref{Maiko}, by the Cauchy--Schwarz inequality and \eqref{N_2quadro} we have
\begin{align*}
\frac{1}{2}\int_\X\|\D_\alpha^2 u_n\|_{\mathcal{H}_\alpha}^2d\nu &\leq \int_\X\langle \D_\alpha f_n,\D_\alpha u_n\rangle_\alpha d\nu=-2\int_\X f_n N_2 u_nd\nu\\
&=-2\int_\X f_n (\lambda u_n-f_n)d\nu\leq 4\int_\X f_n^2d\nu.
\end{align*}
So we get
\[\|\D_\alpha^2 u_n\|_{L^2(\X,\nu;\mathcal{H}_\alpha)}\leq 2\sqrt{2}\norm{f_n}_{L^2(\X,\nu)}.\]
We remark that
\[\|\D^2_\alpha u_n-\D^2_\alpha u_m\|_{L^2(\X,\nu;\mathcal{H}_\alpha)}\leq 2\sqrt{2}\norm{f_n-f_m}_{L^2(\X,\nu)},\]
then $\{\D^2_\alpha u_n\}_{n\in\N}$ is a Cauchy sequence in $L^2(\X,\nu;\mathcal{H}_\alpha)$. By the closability of $(\D_\alpha,\D^2_\alpha)$ in $L^2(\X,\nu)$ it follows that $u\in W_\alpha^{2,2}(\X,\nu)$ and
\[L^2(\X,\nu;\mathcal{H}_2)\text{-}\lim_{n\ra+\infty}\D_\alpha^2 u_n=\D_\alpha^2 u.\]
Therefore
\[\|\D^2_\alpha u\|_{L^2(\X,\nu;\mathcal{H}_\alpha)}=\lim_{n\ra+\infty}\|\D^2_\alpha u_n\|_{L^2(\X,\nu;\mathcal{H}_\alpha)}\leq\lim_{n\ra+\infty}2\sqrt{2} \norm{f_n}_{L^2(\X,\nu)}=2\sqrt{2}\norm{f}_{L^2(\X,\nu)},\]
and $\set{u_n}_{n\in\N}$ converges to $u$ in $W_\alpha^{2,2}(\X,\nu)$. 
\end{proof}

\subsection{Logarithmic Sobolev inequality and further consequences}
Logarithmic Sobolev inequalities are important tools in the study of Sobolev spaces with respect to non-Lebesgue measures. This is due to the fact that they are the counterpart of the Sobolev embeddings which in general fail to hold when the Lebesgue measure is replaced by other measures, as for example the Gaussian one. In this section we also collect some consequences of the logarithmic Sobolev inequality \eqref{logsob}.



Now we are ready to prove that the measure $\nu$ satisfies a logarithmic Sobolev
inequality. The idea of the proof is to apply the Deuschel and Stroock method (see \cite{DS1}).

\begin{proof}[Proof of Theorem \ref{logsob_pro}]
We split the proof in two parts. In the first part we prove the claim when $\varphi$ satisfies some additional conditions and in the second part we show \eqref{logsob} in its full generality.

\noindent\textbf{\emph{Step 1.}} Here we prove \eqref{logsob} with $\varphi$ in $\mathcal{FC}_b^1(\X)$ such that there exists a positive constant $c$ with $c\leq \varphi\leq 1$. To this aim we consider the function
\begin{equation*}
H(t):=\int_\X (P_2(t)\varphi^p)\ln (P_2(t)\varphi^p)d\nu,\qquad t\geq 0.
\end{equation*}
which is well defined thanks to the contractivity and the positivity preserving of $P_2(t)$.

Our aim is to find a lower bound for the derivative of $H$. Observe that by the invariance of $\nu$ and \eqref{Hazel} we have
\begin{align*}
H'(t)&=\int_\X (N_2 P_2(t)\varphi^p)\ln (P_2(t)\varphi^p)d\nu+\int_\X N_2 P_2(t)\varphi^pd\nu\\
&=-\int_\X\frac{1}{P_2(t)\varphi^p}\|\D_\alpha P_2(t)\varphi^p\|_\alpha^2d\nu\geq -e^{-2\zeta_\alpha t}\int_\X\frac{1}{P_2(t)\varphi^p}P_2(t)\|\D_\alpha \varphi^p\|_\alpha^2d\nu\\
&\geq -e^{-2\zeta_\alpha t}\int_\X\frac{1}{P_2(t)\varphi^p}(P_2(t)\|\D_\alpha \varphi^p\|_\alpha)^2d\nu
\end{align*}
By \eqref{Anya} we have $P_2(t)\norm{\D_\alpha \varphi^p}_\alpha\leq [P_2(t)(\norm{\D_\alpha \varphi^p}^2_\alpha \varphi^{-p})]^{1/2}\pa{P_2 (t)\varphi^p}^{1/2}$. Hence we deduce
\begin{align*}
H'(t)&\geq -e^{-2\zeta_\alpha t}\int_\X P_2(t)\frac{\norm{\D_\alpha \varphi^p}^2_\alpha}{\varphi^p} d\nu\\
&= -e^{-2\zeta_\alpha t}\int_\X \frac{\norm{\D_\alpha \varphi^p}^2_\alpha}{\varphi^p} d\nu=-e^{-2\zeta_\alpha t}p^2\int_\X \varphi^{p-2}\norm{\D_\alpha \varphi}^2_\alpha d\nu.
\end{align*}
Integrating from $0$ to $+\infty$ and by \eqref{pro4} we get
\begin{gather*}
\int_\X\varphi^p\ln \varphi^pd\nu\leq \pa{\int_\X\varphi^pd\nu}\ln\pa{\int_\X\varphi^pd\nu}+\frac{p^2}{2\zeta_\alpha}\int_\X\varphi^{p-2}\norm{\D_\alpha \varphi}_\alpha^2d\nu.
\end{gather*}

\noindent\textbf{\emph{Step 2.}} Now, for any $\varphi\in\mathcal{FC}^1_b(\X)$, consider the sequence $\{\varphi_n\}_{n \in \N}$ defined by $\varphi_n=(1+\norm{\varphi}_\infty)^{-1}\sqrt{\varphi^2+n^{-1}}$. Step 1 yields that
\begin{gather}\label{cable}
\int_\X \varphi_n^p\ln(\varphi_n^p)d\nu\leq \pa{\int_\X \varphi_n^pd\nu}\ln\pa{\int_\X\varphi_n^pd\nu}+\frac{p^2}{2\zeta_\alpha}\int_\X \varphi_n^{p-2}\norm{\D_\alpha \varphi_n}_\alpha^2d\nu.
\end{gather}
Observing that there exists a positive constant $c_{n,p}$ such that $c_{n,p} \le \varphi_n^p\leq 1$ for any $n\in \N$ and using the fact that the function $x\mapsto x\abs{\ln x}$ is bounded in $(0,1]$, by the dominated convergence theorem the left hand side of \eqref{cable} converges to
\[
(1+\norm{\varphi}_\infty)^{-p}\int_\X\abs{\varphi}^p\ln\big[(1+\norm{\varphi}_\infty)^{-p}\abs{\varphi}^p\big]d\nu,
\]
and the first term in the right hand side of \eqref{cable} converges to
\[
\pa{(1+\norm{\varphi}_\infty)^{-p}\int_\X\abs{\varphi}^pd\nu}\ln\pa{\frac{\int_\X|\varphi|^pd\nu}{(1+\norm{\varphi}_\infty)^{p}}}.
\]
Since $\|\D_\alpha \varphi_n\|_\alpha\leq (1+\norm{\varphi}_\infty)^{-1}\|\D_\alpha \varphi\|_\alpha$ for every $n\in\N$, by the monotone convergence theorem if $p\in[1,2)$, and by the dominated convergence theorem otherwise, we obtain
\begin{gather*}
\lim_{n\ra+\infty}\int_\X{\varphi_n}^{p-2}\norm{\D_\alpha \varphi_n}_\alpha^2d\nu=
(1+\norm{\varphi}_\infty)^{-p} \int_\X\abs{\varphi}^{p-2}\norm{\D_\alpha \varphi}_\alpha^2\chi_{\set{\varphi\neq 0}}d\nu.
\end{gather*}
So the statement follows letting $n$ to infinity in \eqref{cable}. To get \eqref{Oldman} it is enough to use a standard approximation argument and the Fatou lemma.
\end{proof}

The logarithmic Sobolev inequality has several interesting consequences. By \cite{Gro75,Gro93}, the logarithmic Sobolev inequality is equivalent to a hypercontractivity type estimate. For the convenience of the reader we have included the proof in Appendix \ref{MaoMao}.

\begin{prop}\label{hyper_pro}
Assume Hypotheses \ref{hyp1} hold true. Let $t>0$ and $q,r\in(1,+\infty)$ be such that $r\leq (q-1)e^{2\zeta_\alpha t}+1$.
Then the operator $P_q(t)$ maps $L^q(\X,\nu)$ in $L^r(\X,\nu)$ and for every $t>0$ and $\varphi\in L^q(\X,\nu)$ it holds
\begin{gather}\label{hyper}
\norm{P_q(t)\varphi}_{L^r\pa{\X,\nu}}\leq
\norm{\varphi}_{L^q\pa{\X,\nu}}.
\end{gather}
\end{prop}

An interesting consequence of Proposition \ref{hyper_pro} is an improvement of positivity property for the semigroup $P_2(t)$. We would like to thank the anonymous referee for pointing out this corollary.

\begin{cor}
Assume Hypotheses \ref{hyp1} hold true. For any $t>0$ the semigroup $P_2(t)$ is positivity improving, meaning that it maps a $\nu$-a.e. non-negative function in a $\nu$-a.e. positive function.
\end{cor}

\begin{proof}
The proof is classical and we just sketch it. By \cite[Theorem 1.7]{MO13} and the classical reverse H\"older and Minkowski inequalities for $p<1$ it is possible to prove a reverse hypercontratcivity estimate such as that of \cite[Section 2]{BOR82}. This is enough to obtain positivity improving, see \cite[Theorem 2.1]{BOR82}.
\end{proof}

Another classical inequality that follows from \eqref{Azala} is the Poincar\'e inequality.

\begin{proof}[Proof of Theorem \ref{Alita}]
We just show the theorem for $\varphi\in\xi_A(\X)$, the general case follows by a standard approximation argument. Letting $t$ go to infinity in \eqref{3dinotte}, using \eqref{Azala} and the invariance of $\nu$ we get
\begin{align*}
\int_\X|\varphi|^2d\nu-\abs{\int_\X\varphi d\nu}^2 &=\int_0^{+\infty}\int_\X\|\D_\alpha P_2(s)\varphi\|_\alpha^2d\nu ds\\
&\leq \int_0^{+\infty}e^{-2\zeta_\alpha s}\int_\X P_2(s)\|\D_\alpha\varphi\|^2_\alpha d\nu ds\\
&= \pa{\int_0^{+\infty}e^{-2\zeta_\alpha s}ds}\pa{\int_\X \|\D_\alpha\varphi\|^2_\alpha d\nu}\\
&=\frac{1}{2\zeta_\alpha}\int_\X \|\D_\alpha\varphi\|^2_\alpha d\nu.
\end{align*}
Recalling that 
\(\int_\X\abs{\varphi-\int_\X\varphi d\nu}^2d\nu=\int_\X|\varphi|^2d\nu-\abs{\int_\X\varphi d\nu}^2\)
we get the thesis.
\end{proof}

The Poincar\'e inequality has many interesting consequences. Here we just state two of them which are relevant to the study of the semigroup $P_2(t)$ and of its generator $N_2$ in $L^2(\X,\nu)$. We already know about the asymptotic behaviour of the semigroup $P_2(t)$, thanks to \eqref{pro4}. The next result gives us the rate to which the semigroup $P_2(t)\varphi$ converges to $\int_\X\varphi d\nu$ in $L^2(\X,\nu)$ when $t$ goes to infinity.

\begin{cor}
Assume Hypotheses \ref{hyp1} hold true. If $\varphi\in L^2(\X,\nu)$, then
\begin{equation*}
\norm{P_2(t)\varphi-\int_\X\varphi d\nu}_{L^2(\X,\nu)}\leq e^{-\zeta_\alpha t}\norm{\varphi}_{L^2(\X,\nu)}.
\end{equation*}
\end{cor}
\begin{proof}
Let \(G(s):=\int_\X\abs{P_2(s)\varphi-\int_\X\varphi d\nu}^2d\nu\). Using both \eqref{poin} and \eqref{Hazel} we get
\begin{align*}
G'(s)&=\frac{d}{ds}\int_\X \abs{P_2(s)\varphi-\int_\X\varphi d\nu}^2d\nu
=2\int_\X (P_2(s)\varphi)(N_2 P_2(s)\varphi)d\nu
\\
&=-\int_\X\norm{\D_\alpha P_2(s)\varphi}^2_\alpha d\nu\leq -2\zeta_\alpha\int_\X \abs{P_2(s)\varphi-\int_X P_2(s)\varphi d\nu}^2d\nu
\\
&=-2\zeta_\alpha\int_\X \abs{P_2(s)\varphi-\int_\X \varphi d\nu}^2d\nu= -2\zeta_\alpha G(s).
\end{align*}
Thus $G(t)\leq e^{-2\zeta_\alpha t}G(0)$, which means
\begin{align*}
\int_\X\abs{P_2(t)\varphi-\int_\X \varphi d\nu}^2d\nu\leq& e^{-2\zeta_\alpha t}\int_\X\abs{\varphi-\int_\X \varphi d\nu}^2d\nu\\
=& e^{-2\zeta_\alpha t}\int_\X\abs{\varphi}^2d\nu-\abs{\int_\X \varphi d\nu}^2d\nu\\
\leq& e^{-2\zeta_\alpha t}\int_\X\abs{\varphi}^2d\nu.\qedhere
\end{align*}
\end{proof}

The next proposition gives us a spectral gap for the operator $N_2$. We refer to \cite[Proposition 10.5.1]{DA-ZA1} for the proof.

\begin{prop}
If Hypotheses \ref{hyp1} hold true, then \(\sigma(N_2)\setminus\set{0}\subseteq \set{\lambda\in\CC \tc {\rm Re}\lambda\leq -\omega}\).
\end{prop}

\section{Examples}

In this section we will give some examples to which the results of this paper can be applied to.

\subsection{A suitable choice of $A$}\label{exA}
In this subsection we suggest a possible choice of the operator $A$ that simplifies the checking of Hypotheses \ref{hyp1}. Let $\beta\geq 0$ and $A=-(1/2)Q^{-\beta}:Q^{\beta}(\X)\subseteq\X\ra\X$. Clearly, due to this choice, $A$ and $Q^{\alpha}$ commute. Let $\{e_k\}_{k \in\N}$ be an orthonormal basis of $\X$ consisting of eigenvectors of $Q$, and let $\{\lambda_k\}_{k\in\N}$ be the eigenvalues associated with $\{e_k\}_{k \in\N}$. Since $Q$ is a compact and positive operator, there exists $k_0\in\N$ such that $\lambda_{k_0}>0$ and $\lambda_k\leq \lambda_{k_0}$, for any $k\in\N$. Without loss of generality we assume $k_0=1$. Hence, for any $x\in Q^{\beta}(\X)$, we have
\begin{equation}\label{disAX}
\scal{Ax}{x}=\sum_{k=1}^{+\infty}-\frac{1}{2}\lambda^{-\beta}_k \scal{x}{e_k}^2\leq -\frac{1}{2}\lambda^{-\beta}_1\norm{x}^2.
\end{equation}
Moreover, since $Q$ is a compact and positive operator, $\Dom(A)=Q^{\beta}(\X)$ is dense in $\X$, so $A$ generates a strongly continuous and contraction semigroup in $\X$. Let $A_{\alpha}$ be the part of $A$ in $H_\alpha$, we recall that 
\[
\Dom(A_{\alpha}):=\{x\in Q^{\alpha}(\X)\cap Q^{\beta}(\X)\; :\; Ax\in Q^{\alpha}(\X)\}.
\]
By \eqref{disAX}, for any $x\in \Dom(A_{\alpha})$, we have
\begin{equation}\label{disAalfa}
\scal{Ax}{x}_\alpha=\langle Q^{-\alpha}Ax, Q^{-\alpha}x\rangle =\langle AQ^{-\alpha}x,Q^{-\alpha}x\rangle \leq -\lambda^{-\beta}_1\|Q^{-\alpha}x\|^2=-\frac{1}{2}\lambda^{-\beta}_1\norm{x}_\alpha^2.
\end{equation}
Since $Q^{\alpha+\beta}(\X)$ is dense in $\X$ and $Q^{-\alpha}$ is a close operator in $\X$, then $Q^{\alpha+\beta}(\X)$ is dense in $H_\alpha$, moreover $Q^{\alpha+\beta}(\X)\subseteq \Dom(A_{\alpha})$. Hence $A$ generates a strongly continuous and contraction semigroup in $H_\alpha$. So $A$ satisfies Hypotheses \ref{hyp0}.
 
\subsection{Semiconvex perturbation}

Let $U:\X\ra\R$ be a $C^2$ function such that $F=-\D U$ verifies Hypotheses\ref{hyp1} and there exists a constant $v\in\R$ such that, for any $x,y\in\X$,
\begin{align}\label{Amnesia}
\langle \D^2U(x)y,y\rangle \geq v\norm{y}^2.
\end{align}
We consider the stochastic partial differential equation
\begin{gather}\label{Bliss}
\eqsys{
dX(t,x)=\big(AX(t,x)-Q^{2\alpha}\D U(X(t,x))\big)dt+Q^{\alpha}dW(t), & t>0;\\
X(0,x)=x\in \X,
}
\end{gather}
Clearly Hypotheses \ref{hyp1} are verified, and so all the results in this paper hold true. In this case the invariant measure $\nu$ of Proposition \ref{Davied} has the following form
\[
\nu(dx):=\frac{e^{-2U(x)}}{\int_\X e^{-2U(x)}\mu(dx)}\mu(dx),
\]
where $\mu$ is a Gaussian measure with mean $0$ and covariance operator
\[
Q_\infty:=\int^{+\infty}_0e^{tA}Q^{2\alpha}e^{tA}dt.
\]
In particular $\nu$ and $\mu$ are equivalent. We refer to \cite[Section 8.6]{DA-ZA2} for a more in-depth treatment of the relationship of $\nu$ and $\mu$ and the study of their properties.
We remark that functions $U$ satisfying \eqref{Amnesia} are called semiconvex, or $2$-paraconvex, and were introduced in \cite{ROL1} and studied by various authors (see, for example, \cite{AT-AZ1,CAN1,CA-SI1}). With regard to the study of \eqref{Bliss} for convex and semiconvex function we refer to \cite{AD-CA-FE1,AN-FE-PA1,CAP-FER1,CAP-FER2,DA2,DA-LU2,DA-LU3,DA-TU1,DA-ZA1,FER1}.

\subsection{An example in $L^2([0,1],\lambda)$}
Let $\X=L^2([0,1],\lambda)$ where $\lambda$ is the Lebesgue measure. Let $-Q^{-1}$ be the realization in $L^2([0,1],\lambda)$ of the second order derivative with Dirichlet boundary condition. Hence $Q$ is a positive and trace class operator. We choose $A=-(1/2)Q^{-\beta}$ as in Subsection \ref{exA}, so $\lambda_1$ in \eqref{disAX} and \eqref{disAalfa} is equal to $\pi^{-2}$ (see \cite[Chapter 4]{DA1}). We assume that $F=Q^{2\alpha}G$ where $G:\X\ra\X$ is Frechet differentiable and Lipschitz continuous with Lipschitz constant $L$. Hence, for any $x,y\in \X$, we have
\[
\scal{F(x)-F(y)}{x-y}\leq \|Q^{2\alpha}\|_{\mathcal{L}(\X)}\norm{G(x)-G(y)}\norm{x-y}\leq \frac{1}{\pi^{2\alpha}}L\norm{x-y},
\]
and, for any $x\in\X$ and $h\in H_\alpha$, we have
\[
\scal{\D F(x)h}{h}_\alpha\leq \norm{DG(x)}_{\mathcal{L}(\X)}\norm{h}^2\leq L\norm{Q^{\alpha}}_{\mathcal{L}(\X)}^2\norm{h}^2_\alpha\leq \frac{1}{\pi^{2\alpha}}L\norm{h}^2_\alpha.
\]
So, if $L<(1/2)\pi^{2\alpha+\beta}$ then the Hypotheses of Theorems \ref{Stime}, \ref{logsob_pro} and \ref{Alita} are verified.

\vspace{0.6cm}

\noindent {\bf Acknowledgements.} The authors would like to thank A. Lunardi for many useful discussions and comments. 

\vspace{0.6cm}

\noindent {\bf Fundings.} The authors are members of GNAMPA (Gruppo Nazionale per l'Analisi Matematica, la Probabilit\`a e le loro Applicazioni) of the Italian Istituto Nazionale di Alta Ma\-te\-ma\-ti\-ca (INdAM).  S. F. have been partially supported by the INdAM-GNAMPA
Project 2019 ``Metodi analitici per lo studio di PDE e problemi collegati in dimensione infinita''. The authors have been also partially supported by the research project PRIN 2015233N5A ``Deterministic and stochastic evolution equations'' of the Italian Ministry of Education, MIUR.

\appendix

{\footnotesize \section{Proofs of the Propositions \ref{derivazmild}, \ref{stisti} and \ref{hyper_pro}}\label{MaoMao}

\begin{proof}[Proof of Proposition \ref{derivazmild}]
Due to some technical difficulties it is convenient to split the proof in two cases. If $m=1$, then $F$ is Lipschitz continuous, so the thesis follows from \cite[Theorem 9.8]{DA-ZA4}. We just need to prove Proposition \ref{derivazmild} when $m>1$. Throughout the proof we let $m':=m(m-1)^{-1}$.

We recall that by Proposition \ref{momemild} the mild solution of \eqref{eqFO} belongs to every space $\X^p([0,T])$ with $p\geq 2$. For any $t\geq 0$, $x\in\X$ and $Y\in \X^{mp}([0,T])$, we define
\begin{align*}
K_t(x,Y):=e^{tA}x+\int_0^t e^{(t-s)A}F(Y(s))ds+\int_0^te^{(t-s)A}Q^{\alpha}dW(s).
\end{align*}
Observe that $K_t$ maps $\X\times \X^{mp}([0,T])$ into $\X^p([0,T])$, indeed by the contractivity of $e^{tA}$, Hypotheses \ref{hyp0} and \eqref{cons1}, we have
\begin{align*}
\E[\|K_t(x,& Y)\|^p] =\E\sq{\norm{e^{tA}x+\int_0^t e^{(t-s)A}F(Y(s))ds+\int_0^te^{(t-s)A}Q^{\alpha}dW(s)}^p}\\
&\leq 3^{p-1}\E\sq{\norm{e^{tA}x}^p+T^{p-1}\int_0^t \norm{e^{(t-s)A}F(Y(s))}^p ds +\norm{\int_0^te^{(t-s)A}Q^{\alpha}dW(s)}^p}\\
&\leq 3^{p-1}\norm{x}^p+3^{p-1}T^{p-1}\int_0^T\E[(1+\|Y(s)\|^m)^p]ds+3^{p-1} \sup_{t\in [0,T]}E[\norm{W_A(t)}^p]\\
&\leq 3^{p-1}\norm{x}^p+3^{p-1}T^{p-1}\int_0^T\E[2^{p-1}+2^{p-1}\|Y(s)\|^{mp}]ds+3^{p-1} \sup_{t\in [0,T]}E[\norm{W_A(t)}^p]\\
&\leq 3^{p-1}\norm{x}^p+6^{p-1}T^{p-1}+6^{p-1}T^p\|Y\|^{mp}_{\X^{mp}([0,T])}+3^{p-1} \sup_{t\in [0,T]}E[\norm{W_A(t)}^p].
\end{align*}
We want to prove that $K_t$ admits partial derivatives along the directions of $\X$ and along the directions of $\X^{mp}([0,T])$. It is easy to check that the partial derivative of $K_t$ in $(x,Y)\in \X\times \X^{mp}([0,T])$ along $y\in\X$ is $e^{tA}y$. We claim that the partial derivative of $K_t$ in $(x,Y)\in\X\times \X^{mp}([0,T])$ along $Z\in \X^{mp}([0,T])$ is
\begin{align*}
L(Z):=\int_0^te^{(t-s)A}F_x(Y(s))Z(s)ds,
\end{align*}
where, for the sake of brevity, $F_x:=\J F$. The function $L: \X^{mp}([0,T])\ra \X^p([0,T])$ is well defined, indeed by the contractivity of $e^{tA}$, Hypotheses \ref{hyp1}, the Young and Jensen inequalities, we have
\begin{align*}
\E[\norm{L(Z)(t)}^p]&=\E\sq{\norm{\int_0^te^{(t-s)A}F_x(Y(s))Z(s)ds}^p}\\
&\leq T^{p-1} \E\sq{\int_0^T \norm{e^{(t-s)A}F_x(Y(s))Z(s)}^pds}\\
&\leq T^{p-1} \E\sq{\int_0^T \norm{F_x(Y(s))Z(s)}^pds}\\
&\leq KT^{p-1} \E\sq{\int_0^T (1+\norm{Y(s)}^{m-1})^p\norm{Z(s)}^pds}\\
&\leq KT^{p-1} \E\sq{\int_0^T (2^{p-1}+2^{p-1}\norm{Y(s)}^{(m-1)p})\norm{Z(s)}^{p}ds}\\
&= K(2T)^{p-1} \E\sq{\int_0^T \norm{Z(s)}^pds}+ K(2T)^{p-1} \E\sq{\int_0^T \norm{Y(s)}^{(m-1)p}\norm{Z(s)}^pds}\\
&\leq \frac{K}{m}2^pT^{p-1} \E\sq{\int_0^T \norm{Z(s)}^{mp}ds}+\frac{K}{m'}2^{p-1}T^p+\frac{K}{m'}(2T)^{p-1} \E\sq{\int_0^T \norm{Y(s)}^{mp}ds}\\
&\leq \frac{K}{m}(2T)^{p} \norm{Z}_{\X^{mp}([0,T])}^{mp}+\frac{K}{m'}2^{p-1}T^p\pa{1+\norm{Y}_{\X^{mp}([0,T])}^{mp}}.
\end{align*}
Now we are ready to show that $L_2(Z)$ is the partial derivative of $K_t$ in $(x,Y)$ along $Z$. Let $\sigma>0$, $x\in X$ and $Y,Z\in \X^{mp}([0,T])$ and consider 
\begin{align*}
I_\sigma(t)&:=\frac{1}{\sigma}\pa{K_t(x,Y+\sigma Z)-K_t(x,Y)-\sigma L(Z)}(t)\\
&=\int_0^t e^{(t-s)A}\sq{\frac{1}{\sigma}\pa{F(Y(s)+\sigma Z(s))-F(Y(s))}-F_x(Y(s))Z(s)}ds.
\end{align*}
So, by the contractivity of $e^{tA}$ and the Jensen inequality, we have
\begin{align*}
\E\Bigg[\int_0^T\|& I_\sigma(t)\|^p dt\Bigg]\\ 
&= \E\sq{\int_0^T\norm{\int_0^t e^{(t-s)A}\sq{\frac{1}{\sigma}\pa{F(Y(s)+\sigma Z(s))-F(Y(s))}-F_x(Y(s))Z(s)}ds}^p dt}\\
&\leq \E\sq{\int_0^T t^{p-1}\int_0^t \norm{e^{(t-s)A}\sq{\frac{1}{\sigma}\pa{F(Y(s)+\sigma Z(s))-F(Y(s))}-F_x(Y(s))Z(s)}}^pds dt}\\
&\leq \E\sq{\int_0^T t^{p-1}\int_0^t \norm{\frac{1}{\sigma}\pa{F(Y(s)+\sigma Z(s))-F(Y(s))}-F_x(Y(s))Z(s)}^pds dt}\\
&=\frac{T^p}{p} \int_0^T\E\sq{\norm{\frac{1}{\sigma}\pa{F(Y(s)+\sigma Z(s))-F(Y(s))}-F_x(Y(s))Z(s)}^p}ds\\
&=\frac{T^p}{p} \int_0^T\E\sq{\norm{\pa{\int_0^1F_x(Y(s)+r\sigma Z(s))Z(s)dr}-F_x(Y(s))Z(s)}^p}ds.
\end{align*}
Now let 
\begin{align*}
J_\sigma:=\norm{\pa{\int_0^1F_x(Y(s)+r\sigma Z(s))Z(s)dr}-F_x(Y(s))Z(s)}^p
\end{align*}
since $F$ satisfies Hypotheses \ref{hyp1}, then $\lim_{\sigma\ra 0}J_\sigma=0$. In order to apply the dominated convergence theorem we need to find an integrable upper bound for $J_\sigma$. By the Young and Jensen inequalities, Hypotheses \ref{hyp1} and the fact that $\sigma\leq 1$, we have
\begin{align*}
J_\sigma &= \norm{\int_0^1F_x(Y(s)+r\sigma Z(s))Z(s)-F_x(Y(s))Z(s)dr}^p\\
&\leq \int_0^1\norm{F_x(Y(s)+r\sigma Z(s))Z(s)-F_x(Y(s))Z(s)}^pdr\\
&\leq \int_0^1\norm{F_x(Y(s)+r\sigma Z(s))-F_x(Y(s))}^p\norm{Z(s)}^pdr\\
&\leq \int_0^1\pa{\norm{F_x(Y(s)+r\sigma Z(s))}+\norm{F_x(Y(s))}}^p\norm{Z(s)}^pdr\\
&\leq \int_0^1\pa{(1+\norm{Y(s)+r\sigma Z(s)}^{m-1})+(1+\norm{Y(s)}^{m-1})}^p\norm{Z(s)}^pdr\\
&\leq \int_0^1\pa{2^{p-1}(1+\norm{Y(s)+r\sigma Z(s)}^{m-1})^p+2^{p-1}(1+\norm{Y(s)}^{m-1})^p}\norm{Z(s)}^pdr\\
&\leq2^{p-1} \int_0^1\pa{(2^{p-1}+2^{p-1}\norm{Y(s)+r\sigma Z(s)}^{(m-1)p})+(2^{p-1}+2^{p-1}\norm{Y(s)}^{(m-1)p})}\norm{Z(s)}^pdr\\
&\leq 2^{2p-1}\norm{Z(s)}^p+2^{2p-2}\norm{Y(s)}^{(m-1)p}\norm{Z(s)}^p+2^{2p-2}\int_0^1\norm{Y(s)+r\sigma Z(s)}^{(m-1)p}\norm{Z(s)}^pdr\\
&\leq 2^{2p-1}\norm{Z(s)}^p+2^{2p-2}(1+\max\{1,2^{(m-1)p-1}\})\norm{Y(s)}^{(m-1)p}\norm{Z(s)}^p\\
&\phantom{000000000000000000000000000000000000000000000}+2^{2p-2}\max\{1,2^{(m-1)p-1}\}\norm{Z(s)}^{mp}\\
&\leq H_1+H_2\norm{Z(s)}^{mp}+H_3\norm{Y(s)}^{mp};
\end{align*}
where 
\begin{align*}
H_1=\frac{2^{2p-1}}{m'},&\qquad H_2=3\frac{2^{2p-2}}{m}+2^{2p-2}\max\{1,2^{(m-1)p-1}\}\pa{1+\frac{1}{m}},\\
&H_3=\frac{2^{2p-2}}{m'}(1+\max\{1,2^{(m-1)p-1}\}).
\end{align*}
So by the dominated convergence theorem the thesis follows.
\end{proof}

\begin{proof}[Proof of Proposition \ref{stisti}]
By Theorem \ref{derivazmild} we have that the map $x\mapsto X(\cdot,x)$ is Gateaux differentiable as a function from $\X$ to $\X^{p}([0,T])$ and, for every $x,h\in\X$, the process $\{\J^G X(t,x)h\}_{t\geq 0}$ is the unique mild solution of
\begin{gather}\label{pasqua}
\eqsys{
\frac{d}{dt}S_x(t,h)=\big(A+\J F(X(t,x))\big)S_x(t,h), & t>0;\\ 
S_x(0,h)=h.
}\end{gather}
Now we scalarly multiply both members of the upper equation of \eqref{pasqua} by $S_x(t,h)$. From the left hand side term we obtain
\begin{align*} 
\scal{\frac{d}{dt}S_x(t,h)}{S_x(t,h)} =\frac{1}{2}\frac{d}{dt}\Vert S_x(t,h)\|^2.
\end{align*}
From right hand side of the above equation, by Hypotheses \ref{hyp0.5} and \eqref{ugo}, we have
\begin{align*}
\scal{ (A+\J F(X(t,x)))S_x(t,h) }{S_x(t,h)}\leq -\zeta\norm{S_x(t,h)}^2.
\end{align*}
Hence we obtain $\frac{d}{dt}\Vert S_x(t,h)\|^2\leq -2\zeta\norm{S_x(t,h)}^2$, and so by the Gronwall inequality $\norm{S_x(t,h)}^2\leq e^{-2wt}\norm{h}^2$.
\end{proof}

\begin{proof}[Proof of Proposition \ref{hyper_pro}]
Let $\varphi\in\mathcal{FC}^1_b(\X)$, with a positive global infimum, and let $r(t):=(q-1)e^{2\zeta t}+1$. We recall that for $P_q(t)$ acts on functions belonging to $\mathcal{FC}^1_b(\X)$ in the same way as $P_2(t)$. For $s\geq 0$ we set
\[
G(s):=\pa{\int_\X(P_2(s)\varphi)^{r(s)}d\nu}^{1/r(s)}=:\pa{R(s)}^{1/r(s)}
\]
and we prove that $G(s)$ is a non-increasing function. Before proceeding we want to recall that $P_2(s)$ maps $\mathcal{F}C^1_b(\X)$ into $W_{\alpha}^{1,2}(\X,\nu)\cap L^\infty(\X,\nu)$, due to \eqref{transition} and Theorem \ref{Stime}. This guarantees that all the integrals we are going to write in the following calculations are well defined and finite. By \eqref{Hazel} we obtain
\begin{align}\label{333}
R'(s)&=r'(s)\int_\X(P_2(s)\varphi)^{r(s)}\ln(P_2(s)\varphi)d\nu-r(s)(r(s)-1)
\int_\X (P_2(s)\varphi)^{r(s)-2}\norm{\D_\alpha P_2(s)\varphi}_\alpha^2d\nu.
\end{align}
Taking into account \eqref{333}, if we set $u(s):=P_2(s)\varphi$ and we differentiate $G$, we get
\begin{align*}
G'(s) &=\frac{r'(s)}{r(s)\int_\X(u(s))^{r(s)}d\nu}\int_\X(u(s))^{r(s)}\ln(u(s))d\nu\\
&\phantom{000000000000}-\frac{r(s)-1}{\int_\X(u(s))^{r(s)}d\nu}\int_\X (u(s))^{r(s)-2}\|\D_\alpha u(s)\|_\alpha^2d\nu-\frac{r'(s)}{r^2(s)}\ln\pa{\int_\X(u(s))^{r(s)}d\nu}
\end{align*}
Since $r'(s)\geq0$ we can apply \eqref{logsob} to get
\begin{gather*}
G'(s)\leq (G(s))^{1-r(s)}\pa{\frac{r'(s)}{2\zeta_\alpha}-r(s)+1}\int_\X (P_2(s)\varphi)^{r(s)-2}\norm{\D_\alpha P_2(s)\varphi}_\alpha^2d\nu=0.
\end{gather*}
This proves that $G$ is a decreasing function, namely $G(0)\geq G(t)$ for every $t>0$. So we have for every $r\leq r(t)$ and $\varphi\in \mathcal{FC}^1_b(\X)$
\begin{gather*}
\norm{P_q(t)\varphi}_{L^{r}\pa{\X,\nu}}\leq \norm{P_2(t)\varphi}_{L^{r(t)}\pa{\X,\nu}}\leq \norm{\varphi}_{L^q\pa{\X,\nu}}.
\end{gather*}
By a standard density argument we obtain \eqref{hyper} for a general $\varphi\in L^q\pa{\X,\nu}$.
\end{proof}

}

\end{document}